%% file: kakeyabourbaki4.tex
\documentclass{amsart}
\usepackage{geometry}                
\usepackage{graphicx}
\usepackage{amssymb}
\usepackage{epstopdf,amsmath}
\usepackage{enumitem}

\usepackage{tikz}
\usepackage{geometry}
\usepackage{pgfplots}

\usepgfplotslibrary{colorbrewer}
\pgfplotsset{width=8cm,compat=1.9}

\newtheorem{prop}{Proposition}[section]
\newtheorem{definition}{Definition}[section]
\newtheorem{lemma}[prop]{Lemma}
\newtheorem*{lemma*}{Lemma}

\newtheorem{theorem}[prop]{Theorem}
\newtheorem{conj}[prop]{Conjecture}

\newtheorem*{prop*}{Proposition}
\newtheorem*{theorem*}{Theorem}

\newcommand{\bes}{\beta_{\textrm{sticky}}}

\newcommand{\WW}{\mathbb{W}}

\newcommand{\FF}{\mathbb{F}}
\newcommand{\CC}{\mathbb{C}}
\newcommand{\TT}{\mathbb{T}}

\newcommand{\RR}{\mathbb{R}}

\newcommand{\SSS}{\mathbb{S}}

\newcommand{\KK}{\mathbb{K}}
\newcommand{\VV}{\mathbb{V}}

\DeclareMathOperator{\dist}{dist}

\title{The Kakeya conjecture, after Wang and Zahl}
\author{Larry Guth}

\begin{document}

\maketitle

This is a survey article about the proof of the Kakeya conjecture in three dimensions.   The Kakeya conjecture is a problem about the intersection patterns of thin tubes in Euclidean space.   

A Kakeya set in $\RR^n$ is a set that contains a unit line segment in every direction.   Around 1920,  Besicovitch gave an example of a Kakeya set in $\RR^2$ with arbitrarily small Lebesgue measure.   Around 1970,  Fefferman gave a counterexample to a well-known problem about Fourier multipliers which crucially used Besicovitch's example.   The same idea shows that a number of open problems in Fourier analysis are connected to Kakeya sets.   These problems in Fourier analysis connect to quantitative questions about Kakeya sets,  such as,  ``What is the infimal Hausdorff dimension of a Kakeya set in $\RR^n$?''   For example,  the Stein restriction conjecture in Fourier analysis implies that every Kakeya set in $\RR^n$ has Hausdorff dimension $n$.   The connection between Fourier analysis and Kakeya problems is described in the survey article \cite{Tsurvey}. 

The Kakeya conjecture for Hausdorff dimension says that every Kakeya set in $\RR^n$ has Hausdorff dimension $n$.   In the 1980s,  Davies proved that every Kakeya set in $\RR^2$ has Hausdorff dimension 2,  and the proof is only a couple pages.   But proving the conjecture for any $n \ge 3$ is much more difficult.   In \cite{WZ},  Wang and Zahl proved the Kakeya conjecture in dimension 3.   In dimension $n \ge 4$, the conjecture is currently open.

The proof of the Kakeya conjecture builds on important ideas by many people,  including Bourgain, Wolff,  Katz, Laba, Tao, Orponen, and Shmerkin.    The goal of this survey is to give an overview of all the ideas in the proof.

\vskip10pt

{\bf Acknowledgements.} Thanks to Seminaire Bourbaki for the invitation to write this survey.   Thanks to Jacob Reznikov for many of the pictures in this article.   And thanks to the many people with whom I have talked about the Kakeya problem over many years,  including Nets Katz,  Hong Wang,  Joshua Zahl,  Pablo Shmerkin,  Alex Cohen,  and Dima Zakharov.

\section{Statement of main results}

In this survey,  we will not use the language of Hausdorff dimension.   For understanding the proof,  and also for applications in Fourier analysis,  the most useful language is in terms of sets of thin tubes.  

Suppose that $\TT$ is a set of $\delta$-tubes in $\RR^n$ with length 1.    We write $U(\TT)$ for $\cup_{T \in \TT} T$.    One version of the Kakeya conjecture in dimension $n$ says 

\begin{conj} \label{conjkakvol} For every $\epsilon >0$, there is a constant $c(n, \epsilon)$ so that if $\TT$ is a set of $\sim \delta^{-(n-1)}$ $\delta$-tubes in $\RR^n$ in $\delta$-separated directions, then 

$$ | U(\TT) | \ge c(n, \epsilon) \delta^\epsilon.$$

\end{conj}

\noindent Wang and Zahl proved this conjecture in dimension $n=3$. 

In fact, they proved a more general estimate called the Kakeya conjecture with convex Wolff axioms.   This therem roughly says that the only way a set of tubes in $\RR^3$ can overlap a lot is by clustering into convex sets.    If $K \subset \RR^3$ is a convex set,  we define

$$ \TT[K] := \{ T \in \TT: T \subset K \}. $$

\noindent We define the density of $\TT$ in $K$ as

$$ \Delta(\TT, K) = \frac{ \sum_{T \in \TT[K]} |T|}{ |K| }. $$

\noindent The density of $\TT$ in $K$ measures how much the tubes of $\TT$ pack into $K$.   Here is a picture to help illustrate the definition.

\includegraphics[scale=.7]{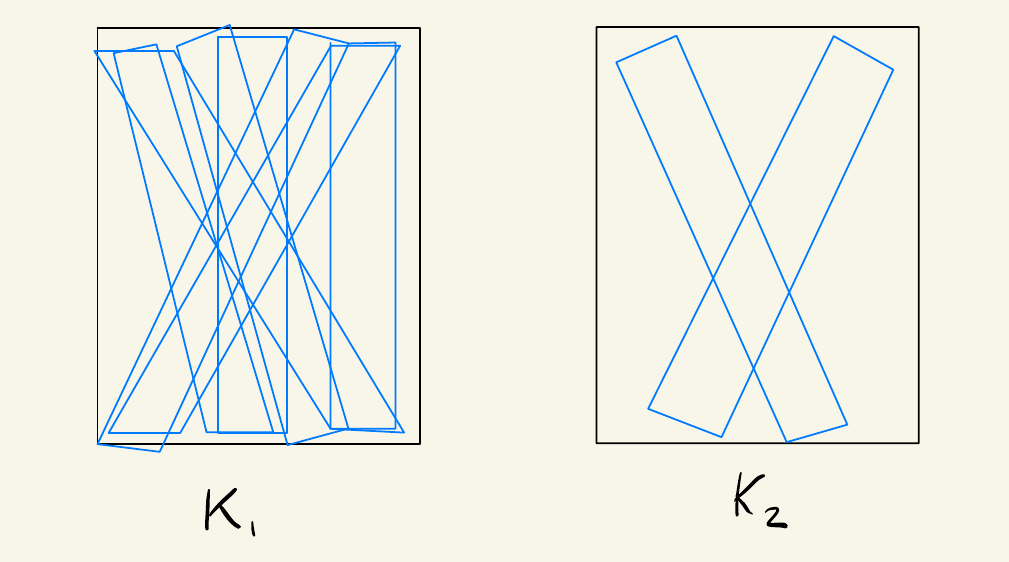}

In this picture, $\Delta(\TT, K_1) > 1$ and $\Delta(\TT, K_2) < 1$.   Next we consider the maximum density over all convex sets $K$. 

$$ \Delta_{max}(\TT) := \max_{K \textrm{ convex}} \Delta(\TT, K). $$

Let us define the typical multiplicity of $\TT$ as

$$ \mu(\TT) = \frac{ \sum_{T \in \TT} |T|}{ | U(\TT) |}. $$

\noindent  On average,  a point $x \in U(\TT)$ lies in $\mu(\TT)$ tubes $T \in \TT$.   Notice that $\mu(\TT[K]) \ge \Delta(\TT,  K)$,  and so there must be a point $x$ that lies in at least $\Delta_{max}(\TT)$ tubes of $\TT$.  

The main theorem of \cite{WZ}  says that $\mu(\TT)$ can only be large when $\Delta_{max}(\TT)$ is large.

\begin{theorem} \label{main} (Wang-Zahl,  \cite{WZ})  If $\TT$ is a set of $\delta$-tubes in $\RR^3$,   and $\Delta_{max}(\TT) \lessapprox 1$,  then

$$ \mu(\TT) \lessapprox 1. $$

\end{theorem}

This theorem is called the convex Wolff axioms version of the Kakeya conjecture.   It directly implies Conjecture \ref{conjkakvol} for $n=3$.   
(Wang and Zahl also proved a somewhat more general theorem which implies that a Kakeya set in $\RR^3$ has Hausdorff dimension 3.)

The proof of Theorem \ref{main} was completed in \cite{WZ}, but the full proof includes many papers with important contributions by Wolff,  Bourgain,  Katz, Laba, Tao, Orponen, and Shmerkin.   The goal of this survey is to describe the main ideas of the whole proof.  

\subsection{Notation}

Informally,  we write $A \approx B$ to mean that $A$ and $B$ are approximately the same size.   We write $A \lessapprox B$ to mean that either $A < B$ or $A \approx B$.   And we write $A \ll B$ to mean that $A$ is much smaller than $B$.   For more detailed statements and outlines,  you can look at \cite{G2} and \cite{GWZ}.

\section{The hero: multiscale analysis}

The hero of our story is looking at the problem at many different scales.  In this section,  we explain what this means and give a hint about why it will be important.

Define $\beta$ to be the infimal exponent so that for every set $\TT$ of $\delta$-tubes in $\RR^3$ with $\Delta_{max}(\TT) \lessapprox 1$, 

\begin{equation} \label{defbeta} \mu(\TT) \lessapprox | \TT |^\beta.
\end{equation}

The Kakeya theorem,  Theorem \ref{main},  says that $\beta = 0$.   We say that $\TT$ is a worst-case Kakeya set if $\Delta_{max}(\TT) \lessapprox 1$ and $\mu(\TT) \approx |\TT|^\beta$.   The proof will be by contradiction.   We suppose $\beta > 0$.   We let $\TT$ be a worst-case Kakeya set.   We will prove that $\TT$ would have to have a lot of geometric and algebraic structure.   Using this structure we will eventually get a contradiction.

To exploit the fact that $\TT$ is a worst-case Kakeya set,  we will compare $\TT$ with other sets of tubes $\TT'$.    We will find other sets of tubes $\TT'$ which are related to $\TT$ and obey $\Delta_{max}(\TT') \lessapprox 1$.   By the definition of $\beta$ in \eqref{defbeta},  we know that $\mu(\TT') \lessapprox |\TT'|^\beta$.   Since $\TT'$ is related to $\TT$,  this bound leads to information about $\TT$.   We will find these sets of tubes $\TT'$ by looking at the original set of tubes at multiple scales.
   
 Suppose that $\TT$ is a set of $\delta$-tubes.  Given a scale $\rho \in [\delta, 1]$,  we let $\TT_\rho$ denote the set of $\rho$-tubes formed by thickening the $\delta$-tubes of $\TT$.   When we thicken two distinct $\delta$-tubes,  we may get nearly identical $\rho$-tubes.   When this happens we identify the $\rho$-tubes.   So we typically have $| \TT_\rho | \ll | \TT |$. 

For each $T_\rho \in \TT_\rho$,  we define

$$ \TT[T_\rho] = \{ T \in \TT: T \subset T_\rho \}. $$

    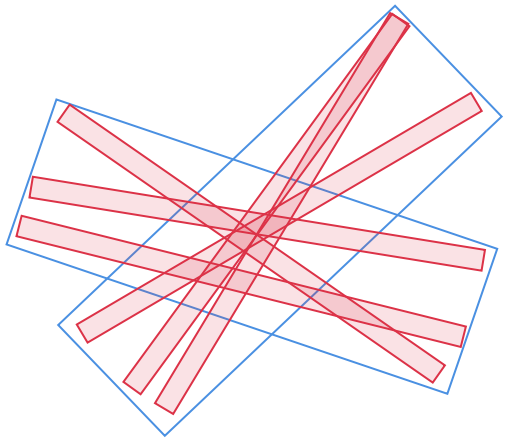
\begin{figure} [h]
      \begin{center}
        \input{intersecting_tubes_diagram}
      \end{center}
      \caption{Intersecting tubes}\label{fig:int_rect}
    \end{figure}

Figure \ref{fig:int_rect} illustrates
these different sets of tubes.   The tubes of $\TT$ are the thin red tubes,  and the tubes of $\TT_\rho$ are the thick blue tubes.   In the picture,  each set $\TT[T_\rho]$ consists of 3 $\delta$-tubes. 

Our original set of tubes $\TT$ is the disjoint union of $\TT[T_\rho]$: 

\begin{equation} \label{T=UTr}  \TT = \bigsqcup_{T_\rho \in \TT_\rho} \TT[T_\rho]. \end{equation}

After some pigeonholing arguments,  we can assume that $|\TT[T_\rho]|$ is roughly constant for all $T_\rho \in \TT_\rho$.   So for any $T_\rho \in \TT_\rho$,  we get

\begin{equation} \label{T=inout} | \TT | \approx |\TT[T_\rho]| |\TT_\rho|. \end{equation}

We will get a lot of information about $\TT$ by appling \eqref{defbeta} to $\TT_\rho$ and $\TT[T_\rho]$ for many different scales $\rho$.   Over the course of the proof,  we will also find other sets of tubes $\TT'$ related to $\TT$ and apply \eqref{defbeta} to them.

 I believe that Tom Wolff was the first person to think about this multi-scale structure in the Kakeya problem,  in unpublished work shortly before his death.   He shared his ideas with Katz, Laba, and Tao, who developed them further in the remarkable paper \cite{KLT},  and then the ideas were developed further by many people.

\section{A key obstacle: the Heisenberg group}

Next we introduce one of the key obstacles to proving the main theorem.   There are cousin problems that sound quite similar to the Kakeya conjecture but behave differently.   For example, the natural analogue of Theorem \ref{main} in $\CC^3$ is false.  The counterexample is called the Heisenberg group example.  It was first published by Katz-Laba-Tao in \cite{KLT}.   The idea that problems of this type may behave differently over different fields was first noted by Tom Wolff in \cite{Wsurvey}.  

We define a metric on $\CC^n$ by identifying it with $\RR^{2n}$.  We define a complex tube in $\CC^n$ with radius $r$ and length $L$  by taking a complex line in $\CC^n$,  intersecting it with a ball of radius $L$,  and then taking its $r$-neighborhood.   We define a complex $\delta$-tube to be a tube of radius $\delta$ and length 1.   The quantities $\Delta(\TT, K)$ and $\Delta_{max}(\TT)$ can be defined roughly as above.   

In $\CC^3$,  there is a set of complex $\delta$-tubes $\TT$ with $\Delta_{max}(\TT) \lessapprox 1$ and $\mu(\TT) \approx |\TT|^{1/4}$.    This example shows that the complex analogue of Theorem \ref{main} is false.  This example is called the Heisenberg group example.   It is based on a quadratic real algebraic hypersurface in $\CC^3$.   There are a couple choices for this hypersurface.  One is the surface $H$ defined by 

$$H = \{ (z_1, z_2, z_3:  |z_1|^2 + |z_2|^2 - |z_3|^2 = 1 \} $$

\noindent  This hypersurface contains many complex lines.   For example,  if $\alpha$ is a unit complex number,  then the line defined by $z_1 = 1$,  $z_3 = \alpha z_2$ lies in $H$.   All these lines pass through the point $(1,0,0) \in H$.   The surface $H$ is very symmetric: it is symmetric under the action of the group $U(2,1)$, which acts transitively on $H$.   By symmetry, there are infinitely many lines through every point of $H$.   Taking $\delta$-neighorhoods of these lines gives a set of $\delta$-tubes $\TT$ where $\Delta_{max}(\TT) \lessapprox 1$ and yet $\mu(\TT) \approx \delta^{-1} \approx |\TT|^{1/4}$.  

In \cite{Whair} in 1995,  Wolff proved that if $\TT$ is a set of $\delta$-tubes in $\RR^3$ with $\Delta_{max}(\TT) \lessapprox 1$,  then $\mu(\TT) \lessapprox | \TT |^{1/4}$.   It has been very difficult to improve on the exponent $1/4$.   In 2001,  in \cite{KLT},  Katz-Laba-Tao improved $1/4$ to $1/4 - \epsilon$ for some tiny $\epsilon > 0$,  under an additional technical assumption.   This technical assumption was removed by Katz-Zahl in \cite{KZ} in 2019.   The exponent $1/4 - \epsilon$ was the best known exponent before the recent work of Wang and Zahl,  giving the sharp exponent.   

Understanding the structure of the Heisenberg group example is essential to the proof of the Kakeya conjecture.  Roughly speaking, the proof shows that if the Kakeya conjecture was false,  a worst-case Kakeya set would need to have structural properties similar to those of the Heisenberg group.  Finally, these strong structural properties lead to a contradiction.  In the next section, we discuss these key structures.

\section{Key structures and outline of the proof}

In this section we introduce the key structures of the Heisenberg group example which guide the proof of the Kakeya conjecture.

\subsection{Grain structure}

Write $N_w(X)$ for the $w$-neighborhood of $X$: 

$$N_w(X) = \{ x \textrm{ so that } \dist(x, X) < w \}. $$

Recall that $H$ is a smooth real 5-manifold in $\RR^6$.  Therefore, if $p \in H$, $N_{\delta} H \cap B_{\sqrt{\delta}}(p)$ is essentially the $\delta$-neighborhood of a 5-dimensional disk.  We describe the situation as follows:

\vskip10pt

{\bf Grain structure.} For each $p \in H$, we can choose unitary coordinates $w_1, w_2, w_3$ on $B_{\sqrt{\delta}}(p)$ so that $N_{\delta} H \cap B_{\sqrt{\delta}}(p) = B^2(\sqrt \delta) \times A$, where 

\begin{itemize}

\item $B^2(\sqrt{\delta})$ is the ball of radius $\sqrt{\delta}$ in $\CC^2$

\item $A$ is the $\delta$-neighborhood of $\RR$ in $B^1(\sqrt{\delta} \subset \CC$.  

\end{itemize}

We will see that if $\beta > 0$,  then a worst-case Kakeya set in $\RR^3$ would have a similar grain structure: for a typical ball $B = B_{\sqrt{\delta}} \subset U(\TT_{\sqrt{\delta}}$,  we can choose coordinates so that $U(\TT) \cap B = [0, \sqrt{\delta}]^2 \times A$,  where $A \subset [0, \sqrt{\delta}]$ is a union of $\delta$ intervals. 

Figure \ref{fig:grains} illustrates the situation.

    \begin{figure} [h!]
      \begin{center}
\includegraphics[scale=.7]{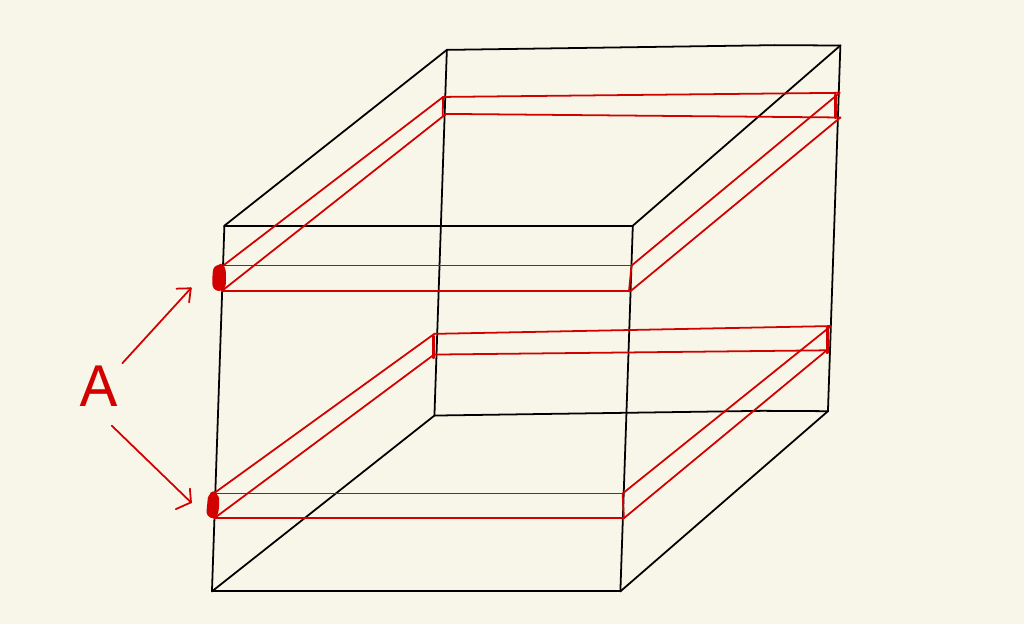}       
      \end{center}
      \caption{Grain structure}\label{fig:grains}
    \end{figure}

The parallel slabs in the picture are parallel to the $(x_1,x_2)$-plane and have thickness $\delta$.    The heights of these slabs correspond to the set $A$. 

The word grain is supposed to evoke the grains in a piece of wood.  Each slab in the picture is called a grain.  We will call $B_{\sqrt{\delta}}(p)$ a grain box.  The grains in a grain box are all parallel to each other.

Each point $z \in H$ lies in a grain, which we call the grain through $z$.

\subsection{Complex conjugation structure}

The definition of $H$ involves complex conjugation, and the structure of $H$ is closely related to complex conjugation.  One connection between complex conjugation and the geometry of $H$ comes from considering the slopes of grains along a line.

Fix a line $\ell \subset H$.   We choose coordinates so that the line $\ell$ is the $z_1$ axis.  Each point $z \in \ell$ lies in a grain, which is a complex 2-plane containing $\ell$.  Such a 2-plane is given by an equation $z_3 = s z_2$, where $s \in \CC$.  Suppose the grain of $H$ through $z = (z_1, 0, 0) \in \ell$ is given by $z_3 = s(z_1) z_2$.  So $s(z_1)$ describes the slope of the grain through $(z_1, 0, 0) \in \ell$.   For an appropriate choice of coordinates $z_1, z_2, z_3$,  the slope function is given by:

\begin{equation} \label{compconj1} s(z_1) = \bar z_1 \end{equation}

We refer to this equation as the complex conjugation structure of $H$.  

\vskip10pt

We can also define a slope function $s(x)$ related to a worst-case Kakeya set in $\RR^3$.  First we will prove that a worst-case Kakeya set in $\RR^3$ has a grain structure as above.  Suppose that $\ell$ is the core line of a tube $T \in \TT$.   We choose coordinates so that the line $\ell$ is the $x_1$-axis.   Each point $(x_1,0,0) \in \ell$ lies in a grain,  which is a plane given by $x_3 = s(x_1) x_2$.   So $s(x_1)$ is the slope of the grain through through $(x_1,0,0)$.   We will prove that in some sense the function $s(x_1)$ ``is similar to complex conjugation''.   

To get a feeling for what this might mean,  let's return to the Heisenberg group example.   In the Heisenberg group example,  $A$ is essentially $\RR \subset \CC$ and the slope function is $s(z_1) = \bar z_1$.   The set $A$ and the slope function $s(z_1)$ interact in a nice way.  For instance, for any $z_1 \in \CC$, 

\begin{equation} \label{naivecompconj} A + s(z_1) z_1 = \RR + |z_1|^2 = \RR = A \end{equation}

For a worst-case Kakeya set,  we will show that the set $A$ and the slope function $s$ also interact in a nice way in a similar spirit to \eqref{naivecompconj}.   We postpone the precise statement to Section \ref{seccompconj}.  It is a little more complicated than \eqref{naivecompconj}, 
involving two different slope functions along two different lines.

\subsection{Stickiness}

If $\TT$ is the Heisenberg group example at scale $\delta$,  then the thicker tubes $\TT_\rho$ are the Heisenberg example at scale $\rho$.   This leads to some nice numerology about $|\TT_\rho|$.   Let $|T^{\CC}_\rho|$ denote the volume of a complex $\rho$-tube in $\CC^3$.   Then for the set $\TT$ of complex tubes from the Heisenberg group,  we have $|\TT| \sim |T^{\CC}_\delta|^{-1}$ and for each $\rho \in [\delta, 1]$,  $|\TT_\rho| \sim |T^{\CC}_\rho|^{-1}$.   This numerology is called the `sticky' case,  for reasons that we explain below.  

\vskip10pt

{\bf Stickiness.}  If $\TT$ is a set of $\delta$-tubes in $\RR^n$,  then $\TT$ is sticky if $|\TT_\rho| \approx |T_\rho|^{-1}$ for each $\rho \in [\delta, 1]$.  

\vskip10pt

In $\RR^n$,  a $\rho$-tube has volume $|T_\rho| \sim \rho^{n-1}$,  and so $\TT$ is sticky if

\begin{equation} \label{eqsticky} |\TT_\rho | \approx \rho^{-(n-1)} \textrm{ for all } \rho \in [\delta, 1]
\end{equation}

Recall from \eqref{T=inout} that $|\TT| \approx |\TT_\rho| |\TT[T_\rho]|$,  and so if $\TT$ is sticky,  then

\begin{equation} \label{eqstickyb} |\TT [T_\rho]  | \approx (\delta/ \rho)^{-(n-1)} \textrm{ for all } \rho \in [\delta, 1], T_\rho \in \TT_\rho
\end{equation}

For comparison, if $\TT$ obeys $\Delta_{max}(\TT) \lessapprox 1$, then we have $| \TT [ T_\rho] | \lessapprox (\delta / \rho)^{-(n-1)}$.  So a set of tubes $\TT$ with $\Delta_{max}(\TT) \lessapprox 1$ is sticky if $|\TT [T_\rho]|$ is as large as possible.  The name sticky comes from the following image.  If two tubes $T_1, T_2$ lie in the same fatter tube $T_\rho$, then they are ``stuck together''.  The tubes in a sticky Kakeya set stick together as much as possible, given the condition $\Delta_{max}(\TT) \lessapprox 1$.  

\vskip10pt

To summarize,  the Heisenberg group example has three important structures: stickiness,  grain structure,  and complex conjugation structure.

\subsection{Outline of the proof}

Now we can outline the proof of the main theorem.
Recall that $\beta$ is the infimal exponent so that for every set $\TT$ of $\delta$-tubes in $\RR^3$ with $\Delta_{max}(\TT) \lessapprox 1$, 

\begin{equation} \mu(\TT) \lessapprox | \TT |^\beta.
\end{equation}

\noindent Our goal is to prove that $\beta = 0$.   We say that $\TT$ is a worst-case Kakeya set if $\Delta_{max}(\TT) \lessapprox 1$ and $\mu(\TT) \approx |\TT|^\beta$.   The proof is by contradiction.   We suppose $\beta > 0$.   We let $\TT$ be a worst-case Kakeya set.

\begin{itemize}

\item Step 1.   Since $\TT$ is worst-case it must be sticky.

\item Step 2.   Then $\TT$ must have grain structure.

\item Step 3.   Then $\TT$ must have something like complex conjugation structure.

\item Step 4.  No such structure exists.  (There is no operation on $\RR$ which has properties similar to complex conjugation.)

\end{itemize}

This is the logical order of the proof but it is not the chronological order.  In particular,  Step 1 was the last step to be understood.

We will explain the proof in chronological order.   Here is an outline.

In the first big part of the paper,  we describe the proof of the sticky case of the main theorem.    This proof was outlined by Katz and Tao in unpublished work in the 2000s,  and shared in a blog post \cite{T}.   It has several steps.

\begin{itemize}

\item In Section \ref{secstickgrain}, we show that stickiness leads to grain structure.  (This step is due to Katz-Laba-Tao,  \cite{KLT},  2001.)

\item In Section \ref{seccompconj}, we show that grain structure leads to complex conjugation structure.  (This step is based on unpublished work of Katz-Tao.   The details were carried out in \cite{WZ1}).

\item In Section \ref{seccompconjcontr}, we show that complex conjugation structure leads to a contradiction.  (This step is based on work of many people,  including Bourgain, Katz, Tao, Orponen, Shmerkin,  Wang, Zahl)

\end{itemize}

After we describe the proof of the sticky case we pause to digest and reflect.   

After that,  in Sections \ref{secworststicky} and \ref{secothershapes} and \ref{sechighdens}, we describe the proof that the worst case is sticky.  (This step is based on work of Wang and Zahl \cite{WZ}).

\section{Avoiding technicalities}

To avoid technical details, we will assume that various quantities are uniform.

For instance,  for each $T_\rho \in \TT_\rho$,  we will study the set of tubes $\TT[T_\rho]$.   In general,  for different tubes $T_\rho \in \TT_\rho$,  $|\TT[T_\rho]|$ could be very different.  But we will assume that all the sets $\TT[T_\rho]$ have roughly the same cardinality.   Similarly we will assume that $\mu(\TT[T_\rho])$ is roughly the same for all $T_\rho \in \TT_\rho$.

For each point $x \in U(\TT)$,  we let $\TT_x := \{ T \in \TT: x \in T \}$.   In general,  for different points $x \in U(\TT)$,  $|\TT_x|$ could be very different.   But we will assume that $|\TT_x|$ is roughly the same for all $x \in U(\TT)$.   Recall that we defined $\mu(\TT) = \frac{ \sum_{T \in \TT} |T| }{|U(\TT)|}$,  which we can interpret as the average size of $|\TT_x|$ over $x \in U(\TT)$.    Since we assume that $|\TT_x|$ is roughly constant,  we have

\begin{equation} \label{Txmu} |\TT_x| \approx \mu(\TT) \textrm{ for every } x \in U(\TT) \end{equation}

In this survey,  we sketch the proof of the Kakeya theorem assuming some uniformity of this kind.   The full proof has the same main ideas but there are extra technical details to deal with non-uniformity.   For instance,  if $|\TT_x|$ is very different at different $x \in U(\TT)$,  then we subdivide $U(\TT)$ into subsets where $|\TT_x|$ has different sizes.   Then we have to keep track of all these subsets.   This process is called pigeonholing.  There is non-trivial technical work involved,  but we will not discuss it in this survey of the high level ideas.

When we state lemmas in this survey,  the precise statements require some pigeonholing and/or uniformity hypotheses.   To keep the statements simple,  we leave out these details.   Unfortunately,  it means that the statements of the lemmas here are not completely precise.   But this choice does keep the statements of lemmas simpler.  We hope that it conveys the general strategy to a broader audience.



\section{Stickiness leads to grain structure} \label{secstickgrain}

In this section,  we will consider a worst-case sticky Kakeya set and show that it must have grain structure.

We say that $\TT$ is a sticky Kakeya set of $\delta$-tubes in $\RR^3$ if:

\begin{itemize}

\item $\Delta_{max}(\TT) \lessapprox 1$.

\item For each $\rho \in [\delta, 1]$,  $|\TT_\rho| \approx \rho^{-2}$ and for each $T_\rho \in \TT_{\rho}$, $ | \TT[T_\rho] | \approx (\rho / \delta)^{2}. $

\end{itemize}

We define $\bes$ to be the infimal exponent $\beta$ so that for every sticky Kakeya set of tubes

\begin{equation}
\mu(\TT) \lessapprox |\TT|^{\beta}.
\end{equation}

\begin{theorem} \label{thmsticky} (Sticky Kakeya theorem \cite{WZ1}) $\bes = 0$.   In other words,  for every sticky Kakeya set $\TT$,  $\mu(\TT) \lessapprox 1$.

\end{theorem}

The sticky Kakeya theorem is the the first big part of the proof of the Kakeya theorem.

We say that $\TT$ is a worst-case sticky Kakeya set if $\TT$ is a sticky Kakeya set and  $\mu(\TT) \approx |\TT|^{\bes}$.   The proof of the sticky Kakeya theorem goes by contradiction.   We suppose that $\bes > 0$ and we let $\TT$ be a worst-case sticky Kakeya set.   By examining $\TT$ at different scales,  we will see that it must have a great deal of structure.

\subsection{Perfect overlap} \label{subsecperfect}

 If $\TT$ is a sticky Kakeya set,  then it follows that for each $\rho \in [\delta, 1]$,  $\TT_\rho$ and $\TT[T_\rho]$ are also sticky Kakeya sets.   This fact makes sticky Kakeya sets well suited for multi-scale induction arguments.


By the definition of $\bes$ (or by induction on scales),  we know that 

\begin{equation} \label{multTrho} \mu(\TT_\rho) \lessapprox |\TT_\rho|^{\bes}. \end{equation}

\begin{equation} \label{multTTrho}  \mu(\TT[T_\rho]) \lessapprox |\TT[T_\rho]|^{\bes}. \end{equation}

Now we state a fundamental lemma relating $\mu(\TT)$ with $\mu(\TT_\rho)$ and $\mu(\TT[T_\rho])$.  

\begin{lemma} \label{lemmultmu} $\mu(\TT) \lessapprox \mu(\TT_\rho) \mu(\TT[T_\rho])$.
\end{lemma}

The following picture illustrates the proof:

\includegraphics[scale=1]{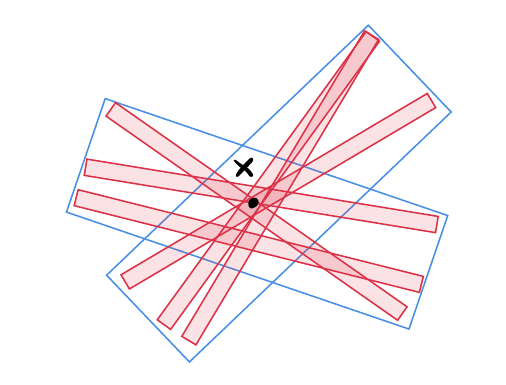}

\begin{proof} [Proof of Lemma \ref{lemmultmu}]
Consider a point $x \in U(\TT)$.    The point $x$ belongs to $T_\rho$ for $\approx \mu(\TT_\rho)$ fat tubes $T_\rho \in \TT_\rho$.   For each of these $T_\rho$,  the point belongs to at most $\mu(\TT[T_\rho])$ thin tubes $T \in \TT[T_\rho]$.  So all together,  $x$ belongs to $\lessapprox \mu(\TT_\rho) \mu(\TT[T_\rho])$ tubes $T \in \TT$.  
\end{proof}

If $\TT$ is a worst-case sticky Kakeya set,  then $|\TT|^{\bes} \approx \mu(\TT)$.   Now combining Lemma \ref{lemmultmu} with our bounds for $\mu(\TT_\rho)$ and $\mu(\TT[T_\rho])$ we see that

\begin{equation} \label{stringeq} |\TT|^{\bes} \approx \mu(\TT) \lessapprox \mu(\TT_\rho) \mu(\TT[T_\rho]) \lessapprox |\TT_\rho|^{\bes} |\TT[T_\rho]|^{\bes} \approx |\TT|^{\bes}.  \end{equation}

\noindent Therefore,  all the inequalities in the above string must be roughly equalities.   This has two important consequences.   Whenever $\TT$ is a worst-case sticky Kakeya set,  we see that

\begin{enumerate}

\item  $\mu(\TT_\rho) \approx |\TT_\rho|^{\bes}$ and $\mu(\TT[T_\rho]) \approx |\TT[T_\rho]|^{\bes}$.   Therefore,  both $\TT_\rho$ and $\TT[T_\rho]$ are worst-case sticky Kakeya sets.

\item  $\mu(\TT) \approx \mu(\TT_\rho) \mu(\TT[T_\rho])$,  so we have equality in Lemma \ref{lemmultmu}.   

\end{enumerate}

To digest this second fact,  we return to the picture illustrating Lemma \ref{lemmultmu} and add a couple more points.

    \begin{figure} [h!]
      \begin{center}
\includegraphics[scale=1]{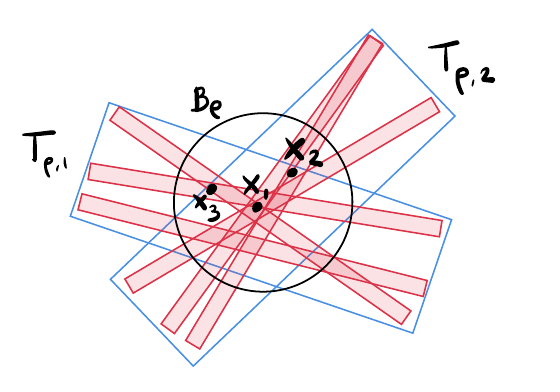}       
      \end{center}
      \caption{When is Lemma \ref{lemmultmu} sharp?}\label{fig:multprodpic}
    \end{figure}

Recall that $\TT_x = \{ T \in \TT: x \in T \}$.    We also define

$$\TT_{\rho,  B_\rho} = \{ T_\rho \in \TT_\rho: B_\rho \cap T_\rho \textrm{ is non-empty} \}. $$

\noindent  The picture shows three points $x_1, x_2, x_3$ all lying in a common $B_\rho$.   In our picture,  $\TT_{\rho, B_\rho}$ is the set of thick blue tubes.    The point $x_1$ lies in $U(\TT[T_\rho])$ for every $T_\rho \in \TT_{\rho, B_\rho }$,  and so

$$ |\TT_{x_1}| \approx |\TT_{\rho,  B_\rho}| |\TT[T_\rho]| \approx \mu(\TT_\rho) \TT[T_\rho]. $$

\noindent For the point $x_1$,  Lemma \ref{lemmultmu} would be an equality.

\noindent But $x_2$ and $x_3$ behave differently.   The point $x_2$ lies in $U(\TT[T_{\rho,2}])$ but not in $U(\TT[T_{\rho,1}])$.   Let us imagine that $x_2$ lies in $U(\TT[T_\rho])$ for only a small fraction of $T_\rho \in \TT_{\rho,B_\rho}$.   Then we would have 

$$ |\TT_{x_2}| \ll \mu(\TT_\rho) \mu(\TT[T_\rho]). $$

If Lemma \ref{lemmultmu} is roughly an equality,  then most points $x \in U(\TT)$ must resemble $x_1$.   So for each $T_\rho \in \TT_{\rho,  B_\rho}$,  $U(\TT_\rho) \cap B_\rho$ must be essentially the same as $U(\TT) \cap B_\rho$.   We call this property the perfect overlap property.

\begin{lemma} (Perfect overlap property) If $\TT$ is a worst-case sticky Kakeya set,  and $B_\rho \subset U(\TT_\rho)$,  then for each $T_\rho \in \TT_{\rho,  B_\rho}$,  

$$ |U(\TT[T_\rho]) \cap B_\rho |  \approx |U(\TT) \cap B_\rho|. $$

Morally,  the sets $U(\TT[T_\rho]) \cap B_\rho$ are all the same as $T_\rho$ varies in $\TT_{\rho,  B_\rho}$.  

\end{lemma}

The perfect overlap property is a very strong condition.   If you look back at Figure \ref{fig:multprodpic},  the tubes in the picture fail the perfect overlap property: most points in $U(\TT) \cap B_\rho$ are like $x_2$ or $x_3$ and only a few are like $x_1$.   Recall that we supposed that $\bes > 0$,  and so $|U(\TT)| \ll 1$.   It's not hard to see that since $\TT$ is a worst-case sticky Kakeya set,  $|U(\TT) \cap B_\rho| \ll |B_\rho|$ (and we will prove this in Section \ref{subsecfrac}).
So $U(\TT[T_\rho]) \cap B_\rho$ are all small subsets of $B_\rho$.  There are many different $T_\rho \in \TT_{\rho, B_\rho}$,  so we have many small subsets $U(\TT[T_\rho]) \cap B_\rho$.   According to the perfect overlap property,  all of these small subsets coincide almost exactly.  This is a strong condition and it gives a lot of information about the Kakeya set.

\subsection{Grain structure} \label{subsecgrain}


The perfect overlap property is easier to analyze when $\rho = \sqrt{\delta}$ because of the following property.

\begin{lemma} \label{lempartubes} If $T_1, T_2 \in \TT[T_{\sqrt{\delta}}]$ and $T_1, T_2$ both intersect $B_{\sqrt{\delta}}$,  then $T_1 \cap B_{\sqrt{\delta}}$ and $T_2 \cap B_{\sqrt{\delta}}$ are essentially parallel tubes of length $\sqrt{\delta}$ and radius $\delta$.   More precisely,  there are parallel tubes $S_1, S_2$ with length $\sqrt{\delta}$ and radius $2 \delta$ so that $T_j \cap B_{\sqrt{\delta}} \subset S_j$.
\end{lemma}

We illustrate the situation in Figure \ref{fig:partubes}

    \begin{figure} [h!]
      \begin{center}
\includegraphics[scale=.7]{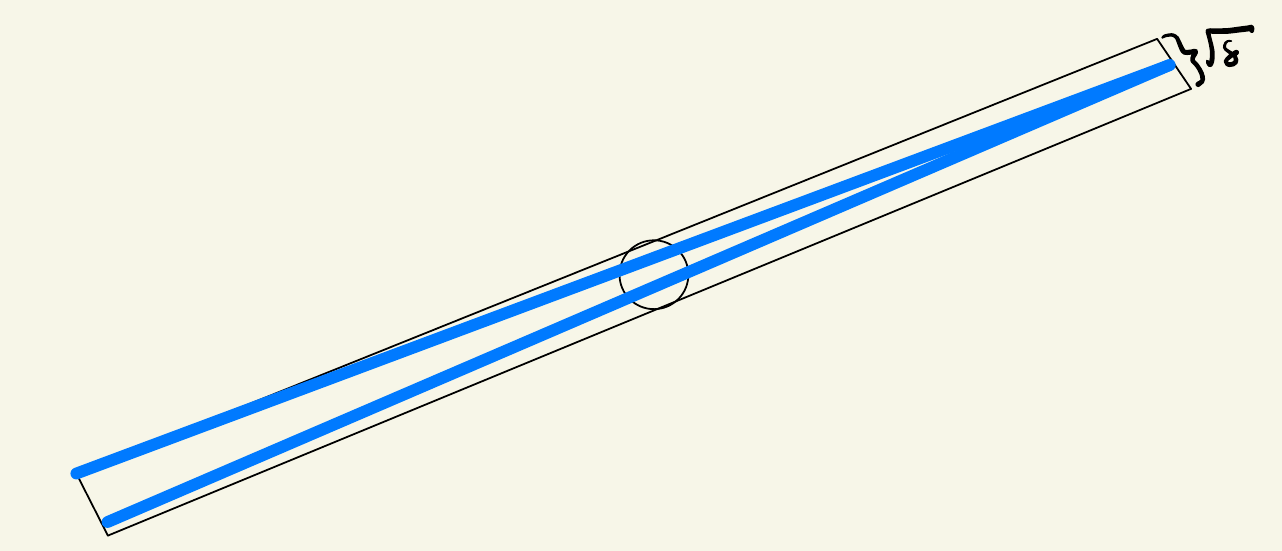}       
      \end{center}
      \caption{Illustration of Lemma \ref{lempartubes}}\label{fig:partubes}
    \end{figure}

It is quite difficult to achieve perfect overlap when $|U(\TT[T_{\sqrt{\delta}}]) \cap B_{\sqrt{\delta}}| \ll |B_{\sqrt{\delta}}|$.   In two dimensions,  this can only happen in the special case when the tubes thru $B_{\sqrt{\delta}}$ all lie in a small angular sector,  as in the following picture.


      \begin{center}
\includegraphics[scale=.7]{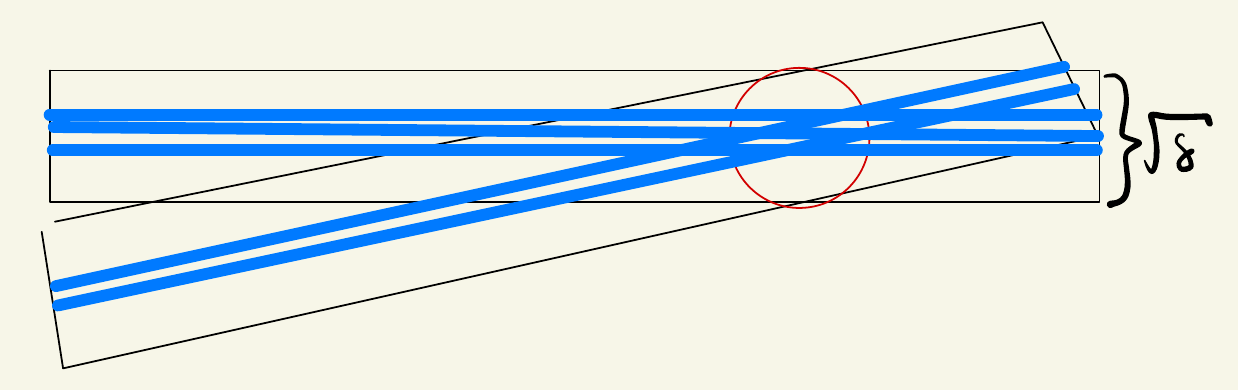}       
      \end{center}

In two dimensions,  if the fat tubes $T_{\sqrt{\delta}}$ through $B_{\sqrt{\delta}}$ are transverse,  then the perfect overlap property implies that $U(\TT)$ fills $B_\rho$.   We state this result as a lemma.

\begin{lemma} \label{lembilin} In two dimensions suppose that $\rho = \sqrt{\delta}$ and 

\begin{itemize}

\item   $T_{\rho, 1}$ and $T_{\rho, 2}$ pass through $B_\rho$.

\item The tubes $T_{\rho, 1}$ and $T_{\rho_2}$ are transverse:  the angle between them is $\sim 1$.

\item $U(\TT[T_{\rho, 1}]) \cap B_\rho = U(\TT[T_{\rho,2}]) \cap B_\rho$,  and these sets are non-empty.

\end{itemize}

Then $| U(\TT[T_{\rho,1}]) ) \cap B_\rho| \approx |B_\rho|$.

\end{lemma}

\begin{proof} [Proof sketch] Let $T_1 \in \TT[T_{\rho,1}]$ be a tube that intersects $B_\rho$.    By hypothesis,  each point $x \in T_1 \cap B_\rho$ lies in a tube $T_{2,x} \in \TT[T_{\rho,2}]$.   These tubes are all parallel to each other,  and they are roughly perpendicular to $T_1$,  and so they fill a definite fraction of $B_\rho$.
\end{proof}  


In three dimensions,  there is a more interesting example that satisfies the perfect overlap property.

{\bf Grain example}  $B = B(0, \sqrt{\delta}) \subset \RR^3$.   

\begin{itemize}

\item $G \subset B$ is a $\delta \times \sqrt{\delta} \times \sqrt{\delta}$ slab parallel to the $(x_1,x_2)$-plane.

\item The tubes of $\TT_{\sqrt{\delta}, B}$ are parallel to the $(x_1,x_2)$-plane.

\item For each $T_{\sqrt{\delta}} \subset \TT_{\sqrt{\delta}}$,  $U(\TT[T_{\sqrt{\delta}}]) \cap B = G$.   

\end{itemize}

Figure \ref{fig:perfectoverlapgrain} illustrates the situation.  The picture can only show some of the tubes.   The tubes running in the $x_1$ direction should fill the slab $G$.  These tubes all lie in a single $T_{\sqrt{\delta}}$ running in the $x_1$ direction.   Similarly,  the tubes running in the $x_2$ direction should fill the slab,  and those tubes all lie in a single $T_{\sqrt{\delta}}$ running in the $x_2$ direction.   We could also add other tubes in any direction parallel to the $(x_1,x_2)$-plane.

    \begin{figure} [h!]
      \begin{center}
\includegraphics[scale=.7]{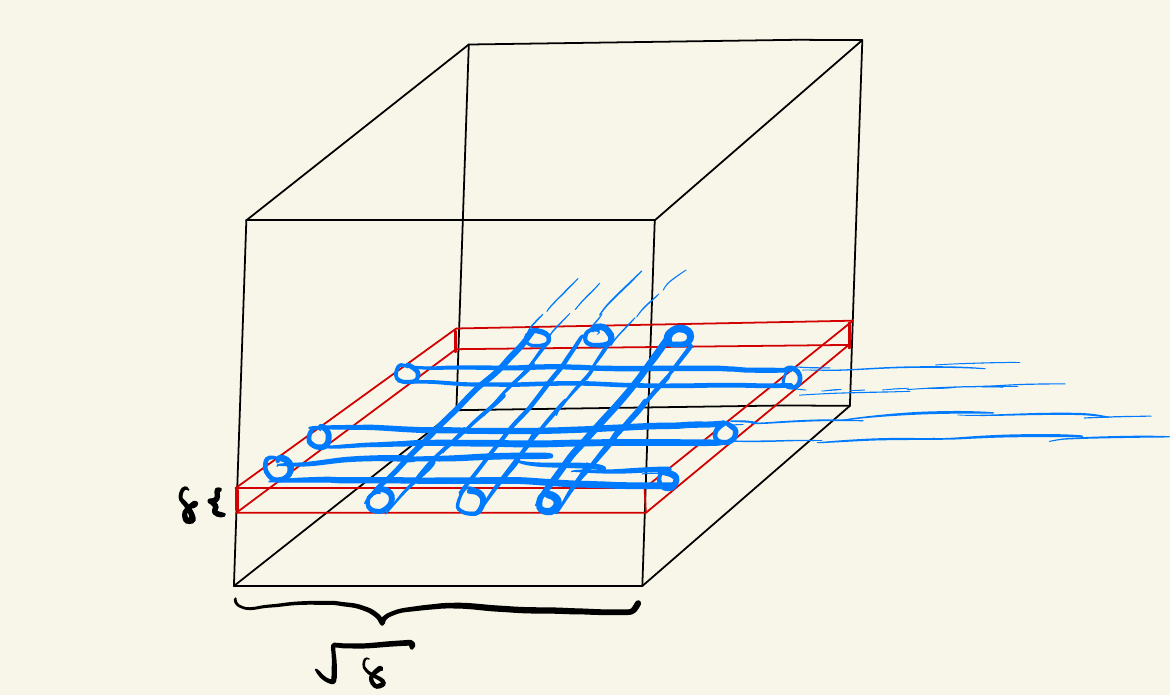}       
      \end{center}
      \caption{Perfect overlap of tubes in a grain}\label{fig:perfectoverlapgrain}
    \end{figure}




In three dimensions,  any example satisfying the perfect overlap property must either be a union of grains or else have all the fat tubes lie in a small angular sector.   It is not hard to check that a worst-case sticky Kakeya set cannot have tubes in such a small angular sector.   This leads to the following grain structure lemma.

\begin{lemma} (Grain structure) Suppose that $\bes > 0$ and that $\TT$ is a worst-case sticky Kakeya set.   Then for a typical ball $B = B_{\sqrt{\delta}} \subset U(\TT_{\sqrt{\delta}})$,  we can choose coordinates so that

\begin{itemize}

\item $U(\TT) \cap B$ is a union of slabs $G$ as in the example above.   Therefore,  $U(\TT) \cap B$ has the form

$$ U(\TT) \cap B = [0, \sqrt{\delta}]^2 \times A,  $$

\noindent where $A \subset [0, \sqrt{\delta}]$.   

\item The tubes of $\TT_{\sqrt{\delta}, B}$ are parallel to the $(x_1,x_2)$-plane.

\end{itemize}

\end{lemma}

The proof of the grain structure lemma is similar to the proof sketch for Lemma \ref{lembilin} above.   We omit the details.

We refer to $B = B_{\sqrt{\delta}} \cap U(\TT_{\sqrt{\delta}})$ as a grain box.   We call each slab $G$ a grain.  For each grain box $B$,  we can choose coordinates so that $B \cap U(\TT) = [0, \sqrt{\delta}]^2 \times A$.   That means that the grains within a grain box are parallel.   But two grains in different grain boxes need not be parallel.

\subsection{Fractal structure of $A$} \label{subsecfrac}

The argument in the perfect overlap section also gives us detailed information about $| U(\TT) \cap B(x,\rho)|$ for any radius $\rho$,  and this gives us important information about $A$.

If $\TT$ is a worst-case sticky Kakeya set,  then we know that $|\TT| \approx \delta^{-2}$ and $\mu(\TT) \approx |\TT|^{\bes} = \delta^{-2 \bes}$.   Since $\mu(\TT) = \frac{ \sum_{T \in \TT} |T|}{|U(\TT)|} \approx \frac{1}{|U(\TT)|}$,  we see that

\begin{equation} \label{volUTst} |U(\TT)| \approx \delta^{2 \bes}
\end{equation}

Now we saw above that if $\TT$ is a worst-case sticky Kakeya set,  then $\TT_\rho$ is also a worst-case sticky Kakeya set for every $\rho \in [\delta, 1]$.   Therefore,  $|U(\TT_\rho)| = \rho^{2 \bes}$.   Now $U(\TT_\rho)$ is the $\rho$-neighborhood of $U(\TT)$.   By uniformity,  we will assume that $|U(\TT) \cap B_\rho|$ is the same for each $B_\rho \subset U(\TT_\rho)$,  and then we get

$$ |U(\TT)| \approx \frac{ |U(\TT) \cap B_\rho|}{|B_\rho|} |U(\TT_\rho)|. $$

Plugging in our values for $|U(\TT)|$ and $|U(\TT_\rho)|$, we see that for each $B_\rho \subset U(\TT_\rho)$, 

$$ |U(\TT) \cap B_\rho| \approx \left( \frac{\delta}{\rho} \right)^{2 \bes} |B_\rho|. $$

If $\rho \le \sqrt{\delta}$,  then $U(\TT) \cap B_\rho$ is described by the grain structure and its geometry depends on the set $A$.   So for any $\rho \in [\delta, \sqrt{\delta}]$ and any interval $I_\rho$ of length $\rho$ centered at a point of $A$,  we get

\begin{equation} \label{fractalA1} |A \cap I_\rho| \approx \left( \frac{ \delta}{\rho} \right)^{2 \bes} \rho. \end{equation}

It is nicest to rewrite this equation in terms of $\delta$-covering numbers.   Recall that the $\delta$-covering number of a set $X$,  written $|X|_\delta$, is the minimal number of $\delta$-balls needed to cover $X$.   Rewriting \eqref{fractalA1} in this language gives:

\begin{equation} \label{fractalA} |A \cap I_\rho|_{\delta} \approx \left( \frac{\rho}{\delta} \right)^{1 - 2 \bes}. \end{equation}

This equation describes the way that the set $A$ is spaced.

Remark.  This type of equation appears in the description of fractals like the Cantor set.   For instance,  if $A$ was the $\delta$-neighborhood of a Cantor set of dimension $1 - 2 \bes$,  it would satisfy this equation.   The technical name for \eqref{fractalA} in the literature is that $A$ is the $\delta$-neighborhood of an AD regular set of dimension $1 - 2 \bes$.



\section{Grain structure leads to complex conjugation structure} \label{seccompconj}

Each tube of $\TT$ enters many different grain boxes, and the grains in these boxes have different slopes.  Let us fix one tube $T_1 \in \TT$ and choose coordinates so that the core line of $T_1$ is the $x$-axis.  For each $x \in [0,1]$, the point $(x,0,0)$ lies in a grain box, and the planes in that grain box are parallel to the $x$-axis.   Therefore,  the grain through $x$ must have the form $z = s(x) y$.   We call $s(x)$ the slope function.

In this section, we will study the function $s(x)$.  Recall that in the Heisenberg group example, in well-chosen coordinates, $x \in \CC$ and $s(x) = \bar x$.  We will see that for a worst-case sticky Kakeya set, the function $s(x)$ has some special properties analogous to properties of complex conjugation.

Recall that if $\TT$ is a worst-case sticky Kakeya set, and if $\rho \gg \delta$, then $\TT[T_\rho]$ is also a worst-case sticky Kakeya set.  So by the grain structure analysis, $U(\TT[T_\rho])$ is also organized into grains.   We can compute the dimensions of these grains by a change of variables argument.    There is a linear change of variables that takes $T_\rho$ to a unit cube and takes $\TT[T_\rho]$ to a set $\tilde \TT$ of $\delta/\rho$-tubes in the unit cube.   According to our grain structure lemma,  $\tilde \TT$ is organized into grain boxes of side length $\sqrt{\delta/\rho}$.  Undoing the linear change of variables, we see that the tubes of $\TT[T_\rho]$ are organized into grain boxes of dimensions $\sqrt{\delta \rho} \times \sqrt{\delta \rho} \times \sqrt{\delta / \rho}$.  We call these long grain boxes, because they are longer and thinner than the regular (cubical) grain boxes.  If $\rho = \sqrt{\delta}$, then these long grain boxes have dimensions $\delta^{3/4} \times \delta^{3/4} \times \delta^{1/4}$.  If $LGB$ is a long grain box, then $U(\TT) \cap LGB$ consists of parallel slabs of dimensions $\delta \times \delta^{3/4} \times \delta^{1/4}$.   We call these slabs long grains. 

The grains in the grain boxes and the long grains in the long grain boxes have to fit together.  In particular, we will fit together two grain boxes and two long grain boxes to form a kind of cycle.  Following the grains around this cycle gives interesting information about the slope function $s(x)$.   Next we explain step by step how the grain boxes and long grain boxes fit together.

We will write $GB$ for a grain box and $G$ for a grain.  We will write $LGB$ for a long grain box and $LG$ for a long grain.  

Let's begin with a grain box around some tube $T_1$.  By choosing coordinates, we can suppose that this box is  $[0, \sqrt{\delta}]^3$.  Because the long grain boxes have height only $\delta^{3/4}$, we will focus on only the bottom part of the box, given by $[0, \sqrt{\delta}]^2 \times [0, \delta^{3/4}]$.  According to the grain structure lemma,  
$U(\TT) \cap [0, \sqrt{\delta}]^2 \times [0, \delta^{3/4}] = [0, \sqrt{\delta}]^2 \times A$, where $A \subset [0, \delta^{3/4}]$.  Here is a picture.


\includegraphics[scale=.7]{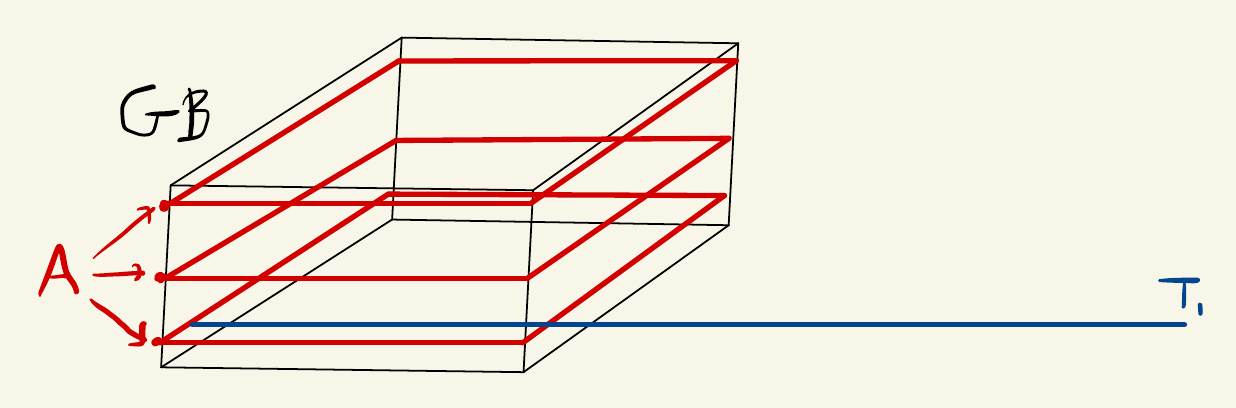}

We're going to have to add more objects to the picture, so for simplicity, we only draw one side of our original grain box.  Here is the abbreviated picture:


\includegraphics[scale=.7]{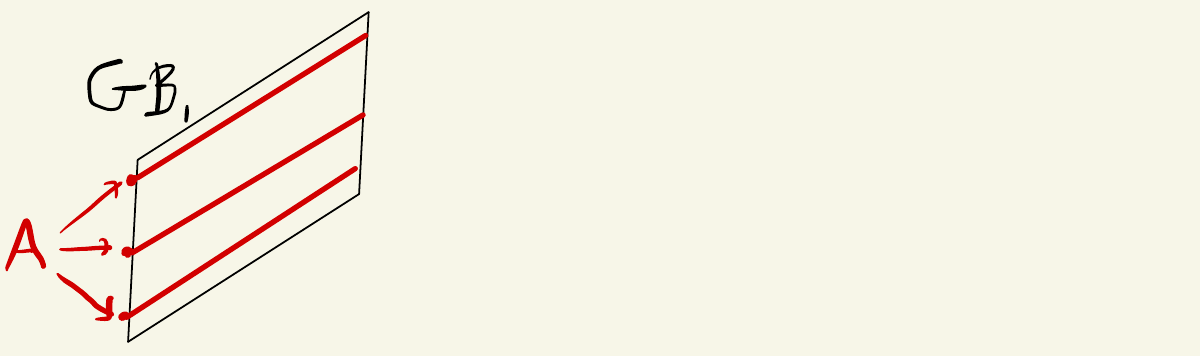}

\noindent Each red line in this picture represents a grain.   We label this first grain box $GB_1$ because we are going to introduce a second grain box soon.   The tube $T_1$ also lies in a long grain box $LGB_1$.   The long grain box $LGB_1$ begins in $GB_1$ but it is much longer than $GB_1$.   Let us add $LGB_1$ to our picture.

\includegraphics[scale=.7]{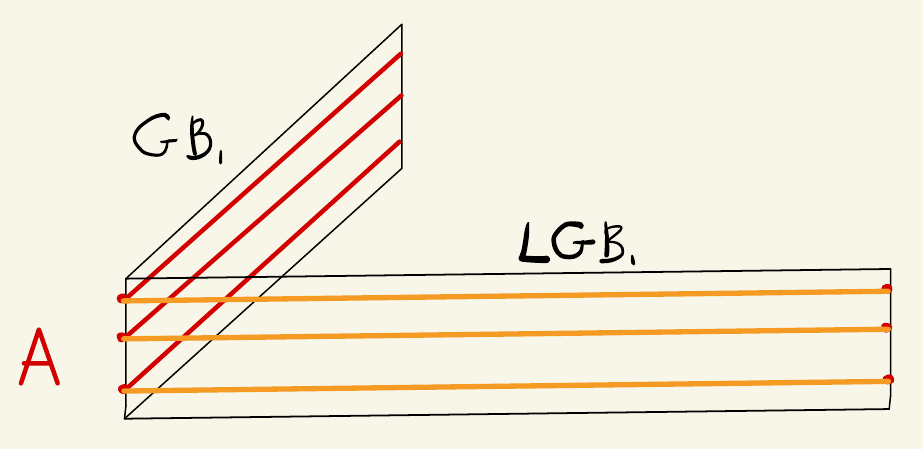}

In this picture,  the bottom rectangle is $LGB_1$,  and the orange horizontal lines in $LGB_1$ represent the long grains.   The long grains have dimensions $\delta \times \delta^{3/4} \times \delta^{1/4}$.    The line in the picture represents the long axis of the long grain.   To keep the picture from being too crowded,  we have left out both of the short axes.   The $\delta^{3/4}$ axis of the long grain runs parallel to the grains of $GB_1$.   (Also, ideally the long grain should be much longer than the grains in $GB_1$,  but it is hard to get so many scales right in the picture.)

The long grain box $LGB_1$ enters many other grain boxes.   If we start at $GB_1$ and follow $LGB_1$ for a distance $x$,  we end up in a second grain box $GB_2$.   The grains in these three grain boxes should fit together as shown in the following picture:

\includegraphics[scale=.7]{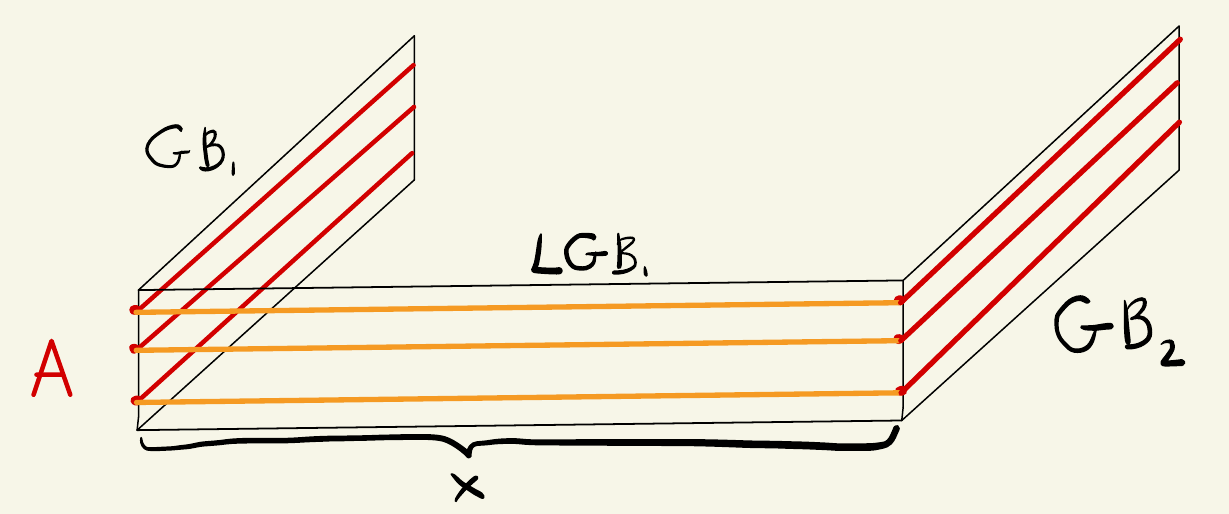}

Recall that the grains in $GB_2$ are all parallel to each other,  but the slope of the grains in $GB_2$ may be slightly different from the slope of the grains in $GB_1$.   We have tried to show this in the picture.   The slope of the grains in $GB_2$ is $s(x)$.

Now there are many different long grains running from $GB_1$ to $GB_2$.   If we follow $GB_2$ for a distance $y$ (in the $y$ direction),  then we end up in a second long grain box $LGB_2$ parallel to the first long grain box $LGB_1$.    Both $LGB_1$ and $LGB_2$ run from $GB_1$ to $GB_2$ and they fit together as in the following picture:

\includegraphics[scale=.7]{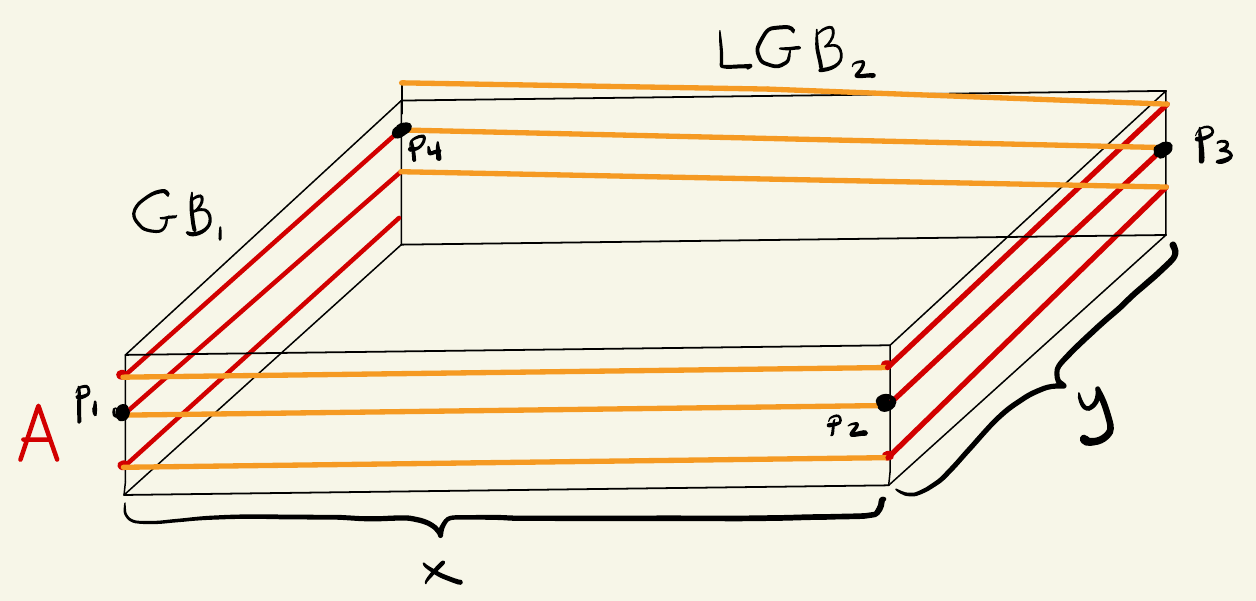}

 This picture reminds me of the ramps in a parking garage.  Starting at some point $p_1$ in the original grain box, we follow a long grain to $p_2$,  then a regular grain to $p_3$,  and then a long grain to $p_4$.   At the end of this cycle,  we have come back to a new ``floor'' in the original grain box.  In this process, no matter what ``floor'' we start on, the $z$ coordinate goes up by a constant amount $\Delta z$.   We can compute $\Delta z$ in terms of the slopes of the different grains and $x$ and $y$.
 
Recall that the grains in $GB_2$ have slope $s(x)$, meaning that the planes are defined by equations of the form $z = s(x) y + c$.  Similarly, let us say that the long grains in the long grain box around $(0,y,0)$ have slope $\tilde s(y)$, meaning that the planes are defined by equations of the form $z = - \tilde s(y) x + c$.  (The negative sign is not important but is convenient.)   So the long grains in $LGB_1$ have slope $\tilde s(0)$ and the long grains in $LGB_2$ have slope $\tilde s(y)$.   We can choose coordinates so that $\tilde s(0) = 0$,  but $\tilde s(y)$ may not be zero.  Now we are ready to compute $\Delta z$.  

Now, if we start at a point $p_1 = (0, 0, z)$ with $z \in A$, we follow the long grain in $LGB_1$ to $p_2 = (x, 0, z + \tilde s(0) x) = (x, 0,z)$.   Then we follow the grain in $GB_2$ to $p_3= (x, y, z + s(x) y)$.   Then we follow the long grain in $LGB_2$ to $p_4 = (0, y, z + s(x) y + \tilde s(y) x)$.  We have now arrived back at a new grain in $GB_1$, and so we should have $z + s(x) y + \tilde s(y) x \in A$.  We have arrived at the following key result:

\begin{lemma} \label{compconjstruclem} If $\TT$ is a worst-case sticky Kakeya set and $A, s(x), \tilde s(y)$ are defined as above, then for  $x \in [0, \delta^{1/4}]$ and $y \in [0, \delta^{1/2}]$, 

\begin{equation} \label{compconjstruceq}
A + s(x) y + \tilde s(y) x \approx A
\end{equation}

\end{lemma}

This result describes a very rigid structure.  Even if we just had a single number $t$ so that $A + t \approx A$, this would encode non-trivial arithmetic structure of the set $A$.  But we have a huge variety of such numbers $t$: every number $t$ that can be written as $s(x) y + \tilde s(y) x$, with $x, y$ at the appropriate scales.

In the Heisenberg group example (over $\CC$),  $A$ would be the real numbers intersected with a ball of an appropriate size, $s(x) = \bar x$ and $\tilde s(y) = \bar y$.  Then Equation \eqref{compconjstruceq} would follow because $\RR + \bar x y + \bar y x = \RR$ for every $x,y \in \CC$.  Equation \eqref{compconjstruceq} captures some of the complex conjugation structure in $\CC$ and so we call \eqref{compconjstruceq} complex conjugation structure.

In this exposition, we have suppressed the difference between things that are true for all points and things that are true for most/many points.  In the full proof this requires technical care.  In particular, the precise statement of Lemma \ref{compconjstruclem} would contain a lot of quantification.  It roughly says that for a large subset of $z \in A$ and a large subset of pairs $(x,y) \in [0, \delta^{1/4}] \times [0, \delta^{1/2}]$, $z + s(x) y + \tilde s(y) x \in A$.

\section{Complex conjugation structure leads to a contradiction} \label{seccompconjcontr}

Equation \eqref{compconjstruceq} captures some of the structure of complex conjugation and also the way that the real numbers fit in the complex numbers.  This is related to several other special features of the way that $\RR$ fits in $\CC$.  The most fundamental feature is that $\RR$ is closed under both addition and multiplication.  Starting from Lemma \ref{compconjstruclem}, the proof of Kakeya shows that there must be a set $A \subset \RR$ which is approximately closed under both addition and multiplication and also obeys spacing estimates as in \eqref{fractalA}.  This area is called sum-product theory, and we give a short introduction to it here.  

Suppose that $A \subset \RR$.  We write $A + A$ for the sumset

$$A + A = \{ a_1 + a_2 | a_1, a_2 \in A \}$$

\noindent and we write $A \cdot A$ for the product set

$$A \cdot A = \{ a_1 a_2 | a_1, a_2 \in A \}.$$  

\noindent If $A \subset \RR$ is a finite set, it is interesting to consider $|A + A|$ and $| A \cdot A|$.  If $A = \{ 1, ..., N \}$, then $|A| = N$ and $|A + A| \sim N$, but $|A  \cdot A | \approx N^2$.   Similarly,  if $A$ is a geometric progression such as $ \{ 2^{n/N} \}_{n=1}^N$, then $|A| = N$ and $|A \cdot A| \sim N$, but $|A + A | \approx N^2$.  Erd{\H o}s conjectured that for any finite set $A \subset \RR$, $|A+A| + |A\cdot A| \gtrapprox |A|^2$.  This conjecture is still open, but we do have some weaker bounds.  Erd{\H o}s and Szemeredi \cite{ES} proved that there is an exponent $\alpha > 0$ so that $|A+A| + |A\cdot A| \gtrapprox |A|^{1 + \alpha}$.  The best known value of $\alpha$ was improved many times, and it is currently a little more than 1/3.

These results are not exactly what is needed in Kakeya type problems.  The set $A$ that appears in our story is not a finite set of points but a finite set of $\delta$-intervals.  Instead of measuring the sizes of finite sets by cardinality, we need to measure the sizes of sets by $\delta$-covering numbers:  if $A \subset \RR$, we write $|A|_\delta$ for the minimum number of $\delta$-balls needed to cover $A$.  Switching our point of view to $\delta$-covering numbers, it is natural to ask: if $A \subset \RR$, then is it true that $|A + A|_\delta + |A \cdot A|_\delta \gtrapprox |A|_\delta^{1 + \alpha}$ for some $\alpha > 0$?  The answer to this question is no.  For instance, this inequality fails if $A = \{ 1 + j \delta \}_{j=1}^J$ whenever $1 \ll J \le \delta^{-1}$.  In order to get a non-trivial sum-product inequality, we need to add an extra assumption saying that $A$ is not packed too tightly into an interval.  The first such theorem was proven by Bourgain in \cite{BSP}, and there is also a closely related result by Edgar-Miller \cite{EM}.  

\begin{theorem}  \label{boursumprod} (Bourgain ``discretized'' sum-product theorem) Suppose that $0 < s < 1$, and $A \subset [0,1]$ with 

\begin{itemize}

\item $|A|_\delta \approx \delta^{-s}$.

\item If $B(x,r) \subset [0,1]$, then $| A \cap B(x,r) |_{\delta} \lessapprox (r/\delta)^s$

\end{itemize}

Then $|A + A|_\delta + |A \cdot A|_\delta \gtrapprox \delta^{-s - \epsilon}$, where $\epsilon = \epsilon(s) > 0$.  

\end{theorem}

With this background,  we can return to discussing the proof of the sticky Kakeya theorem.   Recall from \eqref{fractalA} that the set $A$ obeys the hypotheses above with $s = 1 - 2 \bes$.    We know by earlier arguments that $\bes \le 1/4$ and so if $\bes > 0$,  then $0 < s < 1$.   Roughly speaking,  we might hope that the complex conjugation structure forces $A+A$ and $A \cdot A$ to be small,  which would give a contradiction.   I think this is morally correct but the actual proof is somewhat more complicated.  

First of all there are some trivial ways that $s(x)$ and $\tilde s(y)$ can satisfy \eqref{compconjstruceq}.   We could have $s(x) = \tilde s(y) = 0$.   Or we could have $s(x) = x$ and $\tilde s(y) = - y$.   The proof in \cite{WZ1} first rules out these possibilities.   For instance,  if $s(x) = 0$,  then it would mean that all the tubes crossing $T_1$ must lie in a common planar slab,  and this would eventually force too many tubes into a planar slab,  violating the hypothesis that $\Delta_{max}(\TT) \lessapprox 1$.    

Starting from the complex conjugation structure,  the proof eliminates these trivial possibilities and then shows that some new set related to $A$ is approximately closed under both addition and multiplication,  which contradicts Theorem \ref{boursumprod}.
The argument has several steps,  and we are not able to sketch all the steps in this survey.   The ideas were developed by Katz,  Tao,  Orponen, Shmerkin,  Wang, and Zahl.    The ideas are related to major recent progress in the field of projection theory. 

\vskip10pt

At the beginning of our discussion,  we mentioned that the Kakeya theorem is hard to prove because the Heisenberg group example shows that the analogous statement over $\CC$ is false.   The Heisenberg group example is sticky,  and so the complex analogue of the sticky Kakeya theorem is also false.   Any proof of the Kakeya theorem or sticky Kakeya theorem must include a step that distinguishes $\RR$ from $\CC$.   The sum-product theorem,  Theorem \ref{boursumprod},  is the crucial step that distinguishes $\RR$ from $\CC$.   

Let us quickly note that the analogue of Theorem \ref{boursumprod} over $\CC$ is false.    The counterexample comes because $\RR$ is a subring of $\CC$.   More precisely,  we let $A = \RR \cap B(0,1) \subset \CC$.   Then $|A|_\delta \sim |A+A|_\delta \sim |A \cdot A|_\delta$,  and these are all much smaller than $|B(0,1)|_\delta$.   

Since distinguishing $\RR$ from $\CC$ is a crucial part of the proof of Kakeya,  let us sketch how it is done.   The simplest version of the idea takes place over finite fields.   We let $\FF_q$ denote the finite field with $q$ elements.   We write $q = p^r$ with $p$ prime.   If $r \ge 2$,  then $\FF_q$ has non-trivial subfields.   In particular,  $\FF_{p^2}$ is a degree 2 extension of $\FF_p$,  just as $\CC$ is a degree 2 extension of $\RR$.   Proving quantitative bounds that distinguish $\RR$ from $\CC$ is closely related to proving quantitative bounds that distinguish $\FF_p$ from $\FF_{p^2}$.   

We will prove a sum-product type estimate over $\FF_p$ when $p$ is prime.  Our result will involve some sets more complicated than $A+A$ or $A \cdot A$.   We will need the following definitions.

$$ (A \cdot A)^{\oplus 3} = \left\{ a_1 a_2 + a_3 a_4 + a_5 a_6: a_i \in A \right\}. $$

$$\frac{A-A}{A-A} = \left\{ \frac{a_1 - a_2}{a_3 - a_4} : a_i \in A,  a_3 \not= a_4 \right\}.$$

$$ \frac{(A\cdot A)^{\oplus 3}-(A\cdot A)^{\oplus 3}}{A-A} = \left\{ \frac{ b_1  - b_2 }{a_1 - a_2}: b_i \in (A \cdot A)^{\oplus 3},  a_i \in A,  a_1 \not= a_2 \right\}. $$

If $c \in \FF_p$,  then

$$ A + c A  = \{ a_1 + c a_2 : a_i \in A \}. $$

\begin{lemma} \label{lemFpvsFq} Suppose that $p$ is prime and $A \subset \FF_p$.  Then either
    \begin{enumerate}
        \item $\frac{A-A}{A-A}=\FF_p$, or
        \item $\left|\frac{(A\cdot A)^{\oplus 3}-(A\cdot A)^{\oplus 3}}{A-A}\right|\geq|A|^2$.\label{2}
    \end{enumerate}

\end{lemma}

We note that this lemma is not true in all finite fields.   If $q = p^2$ and $A = \FF_p \subset \FF_q$,  then all the complicated sets in the bullet points are equal to $A$,  and they all have cardinality much smaller than $|\FF_q|$ or $|A|^2$.   The proof of this theorem must distinguish $\FF_p$ from $\FF_{p^2}$.

\begin{proof}
First, note that if $c\not\in\frac{A-A}{A-A}$, then $|A+cA|=|A|^2$.     Indeed if $|A + c A|< |A|^2$,  then by the pigeonhole principle,  we could find $a_1,a_2,a_1',a_2'\in A$ with $a_1+ca_2=a_1'+ca_2'$.   But this implies $c=\frac{a_1'-a_1}{a_2-a_2'}\in\frac{A-A}{A-A}$.

    Next, note that if $\frac{A-A}{A-A}\neq\FF_p$ then there is some $b\in\frac{A-A}{A-A}$ with $b+1\not\in\frac{A-A}{A-A}$.  This step is true in $\FF_p$ but it would fail in $\FF_{p^2}$.   

    Now, if $\frac{A-A}{A-A}\neq\FF_p$, then we have $|A + (b+1) A| \ge |A|^2$ with $b \in \frac{A-A}{A-A}$.   Expanding everything out,  we see that 
    
   $$ A + (b+1) A \subset  \frac{(A\cdot A)^{\oplus 3}-(A\cdot A)^{\oplus 3}}{A-A}. $$
   
   This gives the second option above.
\end{proof}

Remark.  Some version of the idea of this proof goes back to the work of Edgar-Miller \cite{EM},  and the argument was adapted by Bourgain-Katz-Tao \cite{BKT} and Garaev \cite{Ga}.   This proof can also be adapted to the setting of $\RR$ and $\CC$,  which was done by Guth-Katz-Zahl in \cite{GKZ}.  
There are multiple proofs of Theorem \ref{boursumprod}.   An adapted version of Lemma \ref{lemFpvsFq} plays a key role in one proof (from \cite{GKZ}).   The proof of Theorem \ref{boursumprod} is technically complicated,  and there is a good recent exposition of this area in \cite{ORSW}.

\section{Leveraging the sum-product theorem at many scales}

Let us pause to digest a surprising feature of the proof of the sticky Kakeya theorem.   A key obstacle in proving sticky Kakeya is that the analogue over $\CC$ is false.   The sum-product theorem (Theorem \ref{boursumprod}) distinguishes $\RR$ from $\CC$.    However,  the bounds in Theorem \ref{boursumprod} are not sharp, and indeed they are very weak.  How can a non-sharp and very weak theorem be used as a key step in the proof of a sharp theorem?

Over the last several years,  there have been a number of sharp results proven using the non-sharp Theorem \ref{boursumprod}.    This body of work has had a major influence in the field of projection theory.    Some of the main contributors are Orponen,  Shmerkin,  Ren,  and Wang.   Roughly speaking,  it is possible to prove strong and sharp estimates by applying Theorem \ref{boursumprod} many times at different scales.

The proof we have reviewed can be written in different ways.   Let us describe a way to rephrase the proof to highlight the way that we are exploiting Theorem \ref{boursumprod} many times at different scales.

We define $\bes(\delta)$ to be the least exponent so that for every sticky Kakeya set of $\delta$-tubes,  $\mu(\TT) \le | \TT |^{\bes(\delta)}$.    In this language,  the sticky Kakeya theorem says that $\lim_{\delta \rightarrow 0} \bes(\delta) = 0$.   We can organize our proof using a key lemma which says that if $\delta_2$ is far smaller than $\delta_1$, then $\bes(\delta_2)$ is smaller than $\bes(\delta_1)$ by a definite amount.   Here is the precise statement of the lemma.

\begin{lemma} For any $\beta > 0$,  there are $\epsilon > 0$ and $K > 0$ so that if $\delta_1 < 1/10$ and $\bes(\delta_1) \ge \beta$ and $\delta_2 \le \delta_1^K$,  then $\bes(\delta_2) < \bes(\delta_1) - \epsilon$.    Also $\epsilon = \epsilon(\beta)$ and $K = K(\beta)$ are continuous in $\beta$.
\end{lemma}

\noindent Iterating this key lemma at many different scales gives the sticky Kakeya theorem.

The proof sketched in the sections above can be slightly adapated to give a proof of the key lemma.    The proof of the key lemma crucially uses the sum-product theorem,  Theorem \ref{boursumprod}.   The exponent in Theorem \ref{boursumprod} is not sharp,  and is only a tiny improvement on a trivial bound.   Partly for this reason,  the exponent $\epsilon = \epsilon(\beta)$ in the key lemma is not sharp and is only a tiny improvement on a trivial bound.   But applying the key lemma many times at different scales,  we get essentially sharp bounds.   In this process,  we are leveraging the sum-product theorem by applying it many times at different scales and getting a small gain each time.

This finishes our sketch of the proof of the sticky Kakeya theorem.

\section{Sticky vs.  not sticky} \label{secstickynot}

The proof of the sticky Kakeya theorem shows a remarkable connection between the sticky case of the Kakeya problem and mathematical structures like the Heisenberg group,  subrings of $\RR$, and sum-product inequalities.   It is certainly interesting mathematics.   But it was not so clear how much progress these results make towards the general Kakeya conjecture.   Is the sticky case a crucial case?  Or is it just a rare special case?

Should we expect the ``worst'' Kakeya set to be sticky?   Analysts have considered many cousins of the Kakeya problem.  For many years,  the worst known example for every cousin problem was sticky.   In \cite{B1} Bourgain considered a variation of the Kakeya problem for curved tubes in $\RR^3$.   In this curved version,  the smallest possible volume of $|U(\TT)|$ is $\sim \delta$,  and the worst-case example is sticky.   In \cite{KLT},  Katz, Laba, and Tao gave the Heisenberg group example,  which showed that the analogue of the Kakeya problem with convex Wolff axioms is false in $\CC^3$.   The Heisenberg group example is sticky.   It seemed plausible that for a broad class of problems of this type,  the worst-case is sticky.



Katz and Tao and others wondered whether the general Kakeya problem could be reduced to the sticky case,  but they didn't see any way to do it.   In 2017,  in \cite{O},  Orponen proved the sticky case of the Falconer conjecture,  another longstanding problem in geometric measure theory which is a kind of cousin of the Kakeya problem.   
This remarkable proof had a lot of influence in the field,  but no one has managed to reduce the general Falconer conjecture to the sticky case. 

\vskip10pt

Then in 2019 in \cite{KZ}, Katz and Zahl found a new cousin of the Kakeya problem,  and gave evidence that the worse case example is not sticky.
  They considered the Wolff axioms version of the Kakeya problem over a different ring -- they replaced $\RR$ by the ring $A = \FF_p[x] / (x^2)$.   The ring $A$ has a natural notion of distance with two distinct length scales.   There is a cousin of the Heisenberg group in $A^3$ and it leads to a counterexample to the analogue of Theorem \ref{main}.   But unlike in $\CC^3$,  the Heisenberg group cousin in $A^3$ is {\it not} sticky.   It appears likely that in $A^3$,  the sticky case of the analogue of Theorem \ref{main} is true,  but the general case is false.

As of a couple years ago,  I was quite pessimistic about reducing the general case of Kakeya to the sticky case.  

\vskip10pt

The first indication that the sticky case may play a crucial role in problems of this type was the solution of the Furstenberg set conjecture by Orponen-Shmerkin and Ren-Wang in 2024.   The Furstenberg set conjecture is a central question in the field called projection theory, which studies the orthogonal projections of sets and measures in $\RR^d$.    In the late 90s,  Tom Wolff identified the Kakeya problem,  the Falconer problem,  and the Furstenberg set problem as cousin problems involving related issues.  In particular, all three conjectures have versions over $\CC$ which are false, with counterexamples related to the subfield $\RR \subset \CC$.
  In 2024,  in \cite{OS2},  Orponen and Shmerkin proved the sticky case of the Furstenberg set conjecture.   Shortly afterwards,  in \cite{RW},  Ren and Wang proved the full Furstenberg conjecture.   
  
 The proof of the Furstenberg conjecture is a major development in the field, and I think it deserves its own survey article to describe.  Some of the key multiscale ideas in the proof of the Kakeya conjecture grew out of this work, developing over multiple papers, including \cite{BSP}, \cite{B2}, \cite{OS1}, \cite{O2}, \cite{SW}, \cite{OSW}, \cite{OS2}, \cite{RW}, \cite{DW}, and \cite{WW}.  
 The proof of the sticky case in \cite{OS2} leverages the sum-product theorem \ref{boursumprod} at many different scales, just like the proof of sticky Kakeya that we saw here.  It also has its own unique issues and features.   The proof of the general case of the Furstenberg conjecture reduces the problem to two cases: the sticky case and a case which is far from sticky,  which they call the semi-well-spaced case.   They resolve the semi-well-spaced case using Fourier analytic methods building on \cite{GSW}.   And they give a short and elegant multiscale argument which reduces the general Furstenberg conjecture to these two cases.

It was quite surprising to me that the general Furstenberg problem reduces to these two cases.  This work gave a hint that the sticky case might be a key case in other problems as well.   Then in 2025 in \cite{WZ}, Wang and Zahl  reduced the general case of the Kakeya problem to the sticky case.  The second main part of our survey describes this reduction.


\section{The $L^2$ method}

Before discussing the reduction to the sticky case,  let us briefly recall the classical $L^2$ method, which we will need in the proof.

If $T_1$ and $T_2$ are two $\delta$-tubes in $\RR^n$ that intersect at a point and the angle between their core lines is $\theta \ge \delta$,  then

\begin{equation} \label{volinttube}
 | T_1 \cap T_2 | \sim \delta^n \theta^{-1}. 
 \end{equation}

If $\TT$ is a set of $\delta$-tubes in $B_1$,  then we can use (\ref{volinttube}) to upper bound

$$ \int_{B_1} | \sum_{T \in \TT} 1_T |^2 = \sum_{T_1, T_2 \in \TT} |T_1 \cap T_2|. $$

Assuming that $| \TT | \approx \delta^{-(n-1)}$ and that $\Delta_{max}(\TT) \lessapprox 1$,  this method gives the sharp bound

$$ \int_{B_1} | \sum_{T \in \TT} 1_T|^2 \lessapprox \delta^{-(n-2)}. $$

\noindent (This bound is sharp when $\TT$ has one tube in each direction and they all go through the origin.)

Combining this $L^2$ bound with Cauchy-Schwarz gives a lower bound on $|U(\TT)|$.   If $\TT$ is a set of $\delta$-tubes in $\RR^2$ with $| \TT | \approx \delta^{-1}$ and $\Delta_{max}(\TT) \lessapprox 1$,  this method shows that $|U(\TT)| \gtrapprox 1$,  which is equivalent to $\mu(\TT) \lessapprox 1$.   This method resolves the Kakeya conjecuture in two dimensions.

In higher dimensions,  while this $L^2$ estimate is sharp,  it does not lead to good information about $|U(\TT)|$ or $\mu(\TT)$. 

On the other hand,  this $L^2$ method also works well for slabs in $\RR^3$.   For instance,  using the same method,  we can prove that if $\SSS$ is a set of $\delta \times 1 \times 1$ slabs in $B_1 \subset \RR^3$ with $| \SSS | \sim \delta^{-1}$ and $\Delta_{max}(\SSS) \lessapprox 1$,  then $\mu(\SSS) \lessapprox 1$ and $|U(\SSS)| \gtrapprox 1$.

This method can handle many questions about tubes in $\RR^2$ and slabs in $\RR^3$.   In the proof sketch below,  we will meet a few problems of this type,  and we will mention that they can be handled by the $L^2$ method.

\section{The worst case is sticky} \label{secworststicky}

We now begin the second big part of the proof of Kakeya: showing that a worst-case Kakeya set is sticky.   This part of the proof was done in \cite{WZ}.   Here we will follow the exposition in \cite{GWZ},  which slightly streamlines the original proof.  We describe the proof in Sections \ref{secworststicky}, \ref{secothershapes}, and \ref{sechighdens}.  


Recall that $\beta$ is the infimal number so that,  whenever $\Delta_{max}(\TT) \lessapprox 1$,  we have

\begin{equation} \label{critexp} \mu(\TT) \lessapprox |\TT|^\beta  \end{equation}

Recall that $\TT$ is a worst-case Kakeya set if $\Delta_{max}(\TT) \lessapprox 1$ and $\mu(\TT) \approx |\TT|^\beta$.   The hypothesis that $\Delta_{max}(\TT) \lessapprox 1$ implies that $|\TT| \lessapprox \delta^{-2}$.    For these notes,  we focus on the case that $| \TT | \approx \delta^{-2}$,  which shows the main ideas and leaves out some technical issues.  

Our goal is to prove that a worst-case Kakeya set must be sticky.  Assuming that $\beta > 0$ and that $\TT$ is not sticky,  we will prove that 

\begin{equation} \label{goalgain} \mu(\TT) \ll |\TT|^\beta
\end{equation}

\noindent (Recall we use the notation $\mu(\TT) \ll  |\TT|^\beta$ to mean that $\mu(\TT)$ is much less than $|\TT|^\beta$).   This inequality contradicts our assumption that $\mu(\TT) \approx |\TT|^\beta$,  and so we conclude that $\TT$ must be sticky.    Then the sticky Kakeya theorem implies that $\beta = 0$.  

Since $|\TT|\sim \delta^{-2}$,  if $\TT$ is not sticky it means that there is some scale $\rho \in [\delta, 1]$ so that $|\TT[T_\rho]| \ll (\rho/\delta)^{-2}$ and  $|\TT_\rho| \gg \rho^{-2}$.   We will focus on the case  that $\TT$ is not-sticky-at-all-scales,  meaning that for every $\rho$ in the range $\delta \ll \rho \ll 1$,  we have

\begin{equation} \label{notstickyall}  | \TT[T_\rho] | \ll (\rho/\delta)^{-2} \textrm{ and } | \TT_\rho | \gg \rho^{-2}. \end{equation}

Assuming \eqref{notstickyall}, we will sketch the proof of \eqref{goalgain}.  

\vskip10pt

In order to use the definition of $\beta$, we need to relate $\TT$ with other sets of tubes.  We will relate $\TT$ to some other set of tubes $\TT'$ with $\Delta_{max}(\TT') \lessapprox 1$ and we use that $\mu(\TT') \lessapprox |\TT'|^\beta$.  In the not sticky case, it is tricky to find helpful sets of tubes $\TT'$.  Recall that in the sticky case, $\TT_\rho$ and $\TT[T_\rho]$ were both sticky Kakeya sets.  In the not sticky case, we still have $\Delta_{max}(\TT[T_\rho]) \lessapprox 1$, but we don't have $\Delta_{max}(\TT_\rho) \lessapprox 1$, so it is harder to make use of $\TT_\rho$ directly.  One of the main new ideas in the proof is a way to find other relevant sets of tubes $\TT'$.  We will see below a couple different clever ways of doing this.


\subsection{Looking at $\TT$ inside a small ball}  \label{subsecsmallball}

Let $\rho$ be an intermediate scale with $\delta \ll \rho \ll 1$.   

To bound $\mu(\TT)$ in the sticky case,  we considered  $\TT[T_\rho]$ and $\TT_\rho$.  In the non-sticky case, Wang and Zahl also consider a new set of tubes formed by intersecting tubes of $\TT$ with a smaller ball $B \subset B_1$.

To set this up,  let's first let's think about how the tubes of $\TT[T_\rho]$ intersect each other.    If $T_1, T_2 \in \TT$, $T_1$ and $T_2$ intersect, and the angle between $T_1, T_2 \approx \rho$, then $T_1 \cap T_2$ is approximately a shorter tube of radius $\delta$ and length $\delta / \rho$.  Therefore, $U(\TT[T_\rho])$ is a union of shorter tubes of this kind.  Each of these short tubes lies in $\approx \mu(\TT[T_\rho])$ long tubes $T \in \TT[T_\rho]$.    Now let $B$ be a ball of radius $r = \delta /\rho$ and consider how these shorter tubes overlap inside of $B$.  Let $\TT[T_\rho]_{B}$ be the set of these shorter tubes in $B$.  So each tube $T_B \in \TT[T_\rho]_{B}$ is a $\delta \times \delta \times \delta/\rho$ tube in $B$ which lies in $\approx \mu(\TT[T_\rho])$ tubes of $\TT[T_\rho]$.    Figure \ref{fig:tube_int_ball} shows a picture.

    \begin{figure} [h!]
      \begin{center}
        \input{tube_intersect_ball_diagram}
      \end{center}
      \caption{Localizing tubes to a ball}\label{fig:tube_int_ball}
    \end{figure}
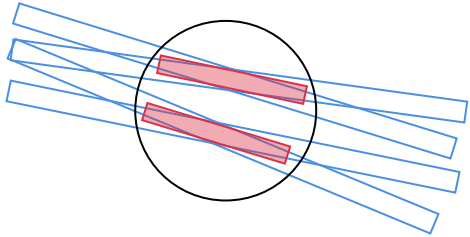

In Figure \ref{fig:tube_int_ball},  the long blue tubes belong to $\TT[T_\rho]$,  the short red tubes belong to $\TT_{T_\rho, B}$,  and the disk is $B$. 

Next we define

$$ \TT_B = \bigcup_{T_\rho \in \TT_\rho,  T_\rho \cap B \not= \emptyset} \TT[T_\rho]_{B}. $$

Figure \ref{fig:tube_int_ball_2} is a picture showing $\TT_B$:

    \begin{figure} [h!]
      \begin{center}
        \input{tube_intersect_ball_two_diagram}
      \end{center}
      \caption{The set of short tubes $\TT_B$}\label{fig:tube_int_ball_2}
    \end{figure}
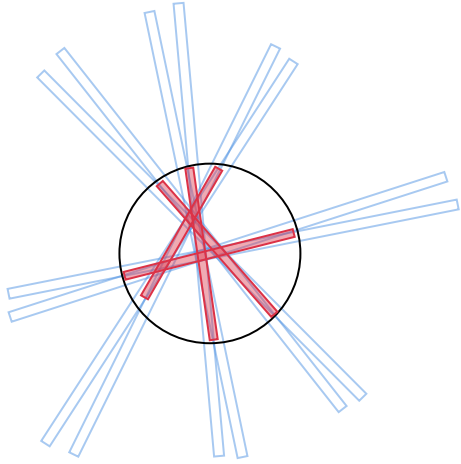

In this picture,  the circle is $B$,  the short red tubes belong to $\TT_B$,  and we see that each tube of $\TT_B$ lies in $\sim \mu(\TT[T_\rho])$ longer tubes of $\TT$.   Therefore,  we can bound $\mu(\TT)$ by 

\begin{equation} \label{twoscales3}  \mu(\TT) \lessapprox \mu(\TT[T_\rho]) \mu(\TT_B).  \end{equation}

We know that $\Delta_{max}(\TT[T_\rho]) \lessapprox \Delta_{max}(\TT) \lessapprox 1$, and so by the definition of $\beta$, $\mu(\TT[T_\rho]) \lessapprox | \TT[T_\rho] |^\beta$.  Since we are in the not-sticky-at-all-scales case (see \eqref{notstickyall}), we also know that  $| \TT[T_\rho]| \ll (\rho / \delta)^{2}$.  So we have

\begin{equation} \label{muTin}  \mu(\TT[T_\rho]) \ll (\rho/\delta)^{2 \beta}. \end{equation}

Next, we have to bound $\mu(\TT_B)$.   Here it is much less clear what to do.   To get started,  let's consider the special case $\Delta_{max}(\TT_B) \lessapprox 1$.  In this special case, we can prove our goal $\mu(\TT) \ll |\TT|^\beta$ by a simple induction argument.

\begin{lemma} \label{lemlowdenscase} If $\TT$ is a worst-case Kakeya set which is not sticky (as in \eqref{notstickyall}),   and IF $\Delta_{max}(\TT_B) \lessapprox 1$,  then $\mu(\TT) \ll |\TT|^\beta$.  
\end{lemma}

\begin{proof} IF  $\Delta_{max}(\TT_B) \lessapprox 1$, then we would get  

$$\mu(\TT_B) \lessapprox |\TT_B|^\beta.$$  

The tubes of $\TT_B$ have radius $\delta$ and length $r = \delta/ \rho$, and so the ratio $\frac{\textrm{radius}(T_B)}{ \textrm{length}(T_B)}= \rho$.  Since $\Delta_{max}(\TT_B) \lessapprox 1$, it follows that $|\TT_B| \lessapprox \rho^{-2}$.  Plugging this bound into the last indented equation, we get

$$ \mu(\TT_B) \lessapprox \rho^{-2 \beta}. $$

Combining this bound with \eqref{muTin}, we get

$$ \mu(\TT)  \lessapprox \mu(\TT[T_\rho]) \mu(\TT_B) \ll (\rho/\delta)^{2 \beta} \rho^{- 2 \beta} = \delta^{-2 \beta } \approx  |\TT|^\beta. $$

\end{proof}

Now the hypothesis that $\TT_B$ has $\Delta_{max}(\TT_B) \lessapprox 1$ is a big IF (that's why I wrote IF in all caps).   The fact that $\Delta_{max}(\TT) \lessapprox 1$ does NOT imply that $\Delta_{max}(\TT_B) \lessapprox 1$.   This makes it a little surprising that Lemma \ref{lemlowdenscase} plays an important role in our proof.

\subsection{A surprising induction}

This is a key philosophical moment in the proof.   We are going to try to control $\TT_B$ using induction.   But the set of tubes $\TT_B$ need not obey $\Delta_{max}(\TT_B) \lessapprox 1$.   This makes it surprising to try to use $\TT_B$ in an inductive proof of Theorem \ref{main}.

To put this moment in context,  let us recall some more history about work on the Kakeya problem.   In \cite{BCT}, Bennett-Carbery-Tao formulated and proved a multilinear cousin of the Kakeya problem.     Their proof was simplified in \cite{G},  and the proof there is only a few pages long.   Multilinear Kakeya involves $n$ sets of tubes $\TT_j$ in $\RR^n$,  where the tubes of $\TT_j$ are approximately parallel to the $x_j$ axis.   Multilinear Kakeya is important because it is much easier than Kakeya but still has many applications - for instance in work of Bourgain-Demeter on decoupling theory \cite{BD}. 

The key feature that makes multilinear Kakeya much easier than Kakeya is that if we intersect the tubes of each $\TT_j$ with a small ball $B$,  then the resulting sets of tubes $\TT_{j,B}$ obey the hypotheses of multilinear Kakeya.   Therefore,  we can easily apply induction to study the intersections of tubes in each small ball $B$.   In contrast,  the hypotheses of the Kakeya probem do not behave well when we restrict the tubes of $\TT$ to a small ball $B$. 

People in the field (including me) tried to adapt the inductive proof of multilinear Kakeya to the original Kakeya problem and we all gave up.   Since the set $\TT_B$ does not obey the hypotheses of the Kakeya conjecture,  how can we control it by induction?

In the Wang-Zahl proof,  we assume nothing at all about the set of tubes $\TT_B$.   But no matter how $\TT_B$ behaves,  they find a way to take advantage of it and prove that $\mu(\TT) \ll | \TT |^\beta$.   There are several scenarios and we will discuss the most important scenarios and how to take advantage of each one.  This proof outline reminds me of a game that was popular in my elementary school called ``heads I win, tails you lose.''

Before turning to details,  let us think through the philosophical issue of how to apply induction to $\TT_B$.    While $\TT_B$ itself does not obey $\Delta_{max}(\TT_B) \lessapprox 1$,  Wang and Zahl manage to relate $\TT_B$ to another set of tubes $ \TT'$ with $\Delta_{max}( \TT') \lessapprox 1$,  and then we can use that $\mu( \TT') \lessapprox | \TT' |^\beta$.   As we go,  we will see how to locate this new set of tubes $ \TT'$ in various scenarios.

\subsection{Organizing the different cases}

If $\Delta_{max}(\TT_B) \lessapprox 1$,  then Lemma \ref{lemlowdenscase} gives us our goal: $\mu(\TT) \ll |\TT|^\beta$.   So we have to consider the case that $\Delta_{max}(\TT_B) \gg 1$.   By definition,  this means that there is some convex set $K \subset B$ so that $ \Delta(\TT_B, K) \gg 1. $   We consider the set $K$ that maximizes $\Delta(\TT_B, K)$.

Notice that for any set $K' \subset K$, $\Delta(\TT_B, K') \le \Delta(\TT_B, K)$.  This condition plays an important role in the story, so we give it a name.

\begin{definition} If $\tilde \TT$ is a set of tubes all contained in a convex set $K$, then $\tilde \TT$ is Frostman in $K$ if for any convex $K' \subset K$, 

$$ \Delta(\tilde \TT, K') \lessapprox \Delta(\tilde \TT, K). $$

\end{definition}

In other words, $\tilde \TT$ is Frostman in $K$ if the tubes of $\tilde \TT$ are contained in $K$ and $\Delta_{max}(\tilde \TT) \approx \Delta(\tilde \TT, K)$.  

Since we picked the set $K$ to maximize $\Delta(\TT_B, K)$,  we see that $\TT_B[K]$ is Frostman in $K$.  (Wang and Zahl picked the name Frostman because the condition is similar to the bound that appears in Frostman's lemma in geometric measure theory.)

Now the proof divides into cases depending on the shape of $K$.   Since $K$ is a convex set which contains some tubes of $\TT_B$, $K$ is essentially a rectangular box of dimensions $a \times b \times r$, where $r$ is the radius of $B$.   The proof divides into cases according to the values of $a$ and $b$.

One important case is when $K = B$.   In this section we focus on this important case.   Then in Section \ref{secothershapes} we consider other possible shapes of $K$.

\subsection{The case $K=B$}

We consider the case when $\Delta_{max}(\TT_B) \approx \Delta(\TT_B, B) \gg 1$.   In this case, $\TT_B$ is Frostman in $B$.

In this case, we first prove a lower bound on $|U(\TT_B)| \approx |U(\TT) \cap B|$ which leads to a lower bound for $|U(\TT)|$, which will imply that $\mu(\TT) \ll |\TT|^\beta$.

(Let us pause here to recall that $U(\TT)$ and $\mu(\TT)$ are closely related to each other: $\mu(\TT) = \frac{ |\TT| |T| }{|U(\TT)|}$.  Using this equation, we can translate bounds on $|U(\TT)|$ into bounds on $\mu(\TT)$ and vice versa.)

The key ingredient in the case $K=B$ is a lower bound for $|U(\tilde \TT)|$ when $\tilde TT$ is a Frostman set of tubes.  We state the lemma as a lower bound for a set of tubes in $B_1$, but we can apply it to $\TT_B$ by rescaling.

\begin{lemma} \label{lemmahighdensity} (High density lemma) Suppose that the exponent $\beta$ is as defined above.   Suppose $\tilde \TT$ is a set of $\delta$ tubes in $B_1$ which is Frostman in $B_1$.   Then

$$ |U(\tilde \TT)| \gtrapprox \left( \delta^2 |\tilde \TT| \right)^{1 - \beta}  \delta^{2 \beta}. $$

\end{lemma}

Let us digest this lemma.  Recall the definition of Frostman: for any convex set $K \subset B_1$, $\Delta(\tilde \TT, K) \lessapprox \Delta(\tilde \TT, B_1)$.  If we let $K$ be one of the tubes of $\tilde \TT$, then we see that $\Delta(\tilde \TT, K) \ge 1$, and therefore $\Delta(\tilde \TT, B_1) \gtrapprox 1$.  This implies that $|\tilde \TT| \gtrapprox \delta^{-2}$.  

In the special case that $|\tilde \TT| \approx \delta^{-2}$, then we have $\Delta(\tilde \TT, B_1) \approx 1$, and so $\Delta_{max}(\tilde \TT) \approx \Delta(\tilde \TT, B_1) \approx 1$.  In this special case, we can apply the definition of $\beta$ to get $\mu(\tilde \TT) \lessapprox |\tilde \TT|^\beta \approx \delta^{-2 \beta}$, and this gives the lower bound $|U(\tilde \TT)| \gtrapprox \delta^{2 \beta}$.  To summarize, if $\tilde \TT$ is Frostman in $B_1$ and $|\tilde \TT| \approx \delta^{-2}$, then $|U(\tilde \TT)| \gtrapprox \delta^{2 \beta}$.  

The main content of Lemma \ref{lemmahighdensity} is that if $|\tilde \TT| \gg \delta^{-2}$, then $|U(\tilde \TT)| \gg \delta^{2 \beta}$.  Equivalently, if $\tilde \TT$ is Frostman in $B_1$ and $\Delta_{max}(\tilde \TT) \gg 1$, then $|U(\tilde \TT)| \gg \delta^{2 \beta}$.  

The proof of Lemma \ref{lemmahighdensity} is complex, and we will discuss it in Section \ref{sechighdens}.

Now we return to $\TT_B$.  We are considering the case when $\TT_B$ is Frostman in $B$ and $\Delta_{max}(\TT_B) \gg 1$.  Rescaling and applying Lemma \ref{lemmahighdensity} to $\TT_B$ gives the bound

\begin{equation} \label{UTBbound} |U(\TT) \cap B| =  |U(\TT_B)| \gg  \left( \frac{\delta}{r}  \right)^{2 \beta} |B|.
\end{equation}

In words, the high density lemma tells us that if $\TT_B$ is Frostman and $\Delta_{max}(\TT_B) \gg 1$,  then $U(\TT)$ fills a surprisingly large fraction of $B$.   We will use this fact to show that $U(\TT)$ is surprisingly large,  which implies that $\mu(\TT) \ll |\TT|^\beta$.

\subsection{Density of the Kakeya set in small balls} \label{subsecdensity}

Upper bounds for the multiplicity of the Kakeya set are closely related to lower bounds for the volume of the Kakeya set.  Recall that we defined the multiplicity by

$$ \mu(\TT) = \frac{ |\TT| |T| }{|U(\TT)|}. $$

Since we are assuming $|\TT| \approx \delta^{-2}$,  

\begin{equation}
\mu(\TT) \ll |\TT|^\beta \textrm{ is equivalent to } |U(\TT)| \gg \delta^{2 \beta}. 
\end{equation}

If $B_r$ is a ``typical'' ball of radius $r$ intersecting the Kakeya set,  then we can write

\begin{equation} \label{prodU} | U(\TT) | \approx | U(\TT_r) | \cdot \frac{ |U(\TT) \cap B_r| }{|B_r|}. \end{equation}

We will refer to $ \frac{ |U(\TT) \cap B_r| }{|B_r|}$ as the density of the Kakeya set in $B_r$.  

Given that $\Delta_{max}(\TT) \lessapprox 1$ and $| \TT | \approx \delta^{-2}$,  it is not hard to show that for all $\delta \le r \le 1$ we have

\begin{equation}   | U(\TT_r) | \gtrapprox r^{2 \beta}. \end{equation}

In the special case when $\TT_B$ is Frostman and $\Delta_{max}(\TT_B) \gg 1$, the high density lemma gives us \eqref{UTBbound}: 

\begin{equation} \label{localvol}   \frac { |U(\TT) \cap B_r|}{|B_r|} \gg  \left( \frac{\delta}{r}  \right)^{2 \beta}. \end{equation}

Putting together the last three indented equations gives $|U(\TT)| \gg \delta^{2 \beta}$, which is equivalent to $\mu(\TT) \ll | \TT |^\beta$.  

\vskip10pt

Let us take stock.  If $\TT$ is not sticky (at all scales), then our goal is to prove that $\mu(\TT) \ll |\TT|^\beta$.  So far we have achieved this goal in two cases:
the case when $\Delta_{max}(\TT_B) \lessapprox 1$ and the case when $\Delta_{max}(\TT_B) \approx \Delta(\TT_B, B) \gg 1$.   Recall that we defined $K \subset B$ to the set that maximizes
$\Delta(\TT_B, K)$.  Our two cases are the case when $K$ is a tube $T_B \in \TT_B$ and the case when $K=B$.   In the next section we discuss other shapes of $K$.

\section{Other possible shapes of $K$} \label{secothershapes}

Suppose that $\Delta_{max}(\TT_B) = \Delta(\TT_B,  K)$.   In general $K$ is a convex set of dimensions $a \times b \times r$.   So far we have discussed the two extreme cases: $K = T_B$ and $K = B$.   Now we turn to various intermediate cases.   By doing some pigeonholing,  we can assume that there is a set $\KK$ of convex sets $K \subset B$,  all with the same dimensions,  so that $\Delta(\TT_B, K) \approx \Delta_{max}(\TT)$ for each $K \in \KK$ and so that each $T_B \in \TT_B$ lies in one $K \in \KK$. 

Essentially this set $\KK$ is formed by using the greedy algorithm.   First we find the convex set $K_1$ which maximizes $\Delta(\TT_B, K_1)$.   We call $K_1$ the densest convex set for $\TT_B$.   Then we remove $\TT_B[K_1]$ from $\TT_B$,  and we let $K_2$ be the densest convex set for the remaining tubes.   We continue in this way until all (or most) of the tubes of $\TT_B$ belong to one of the sets $K \in \KK$.

From the nature of this procedure,  it follows that $\mu(\KK) \lessapprox 1$.   The reason is that if many sets $K \in \KK$ pack into a larger convex set $L$,  then we would have $\Delta(\TT_B, L) \gg \Delta(\TT_B, K)$,  and so the greedy algorithm would have chosen $L$ instead of $K$.  Also, for each $K \in \KK$, $\TT_B[K]$ is Frostman.

The argument divides into cases according to the shape of $K$.  

\vskip10pt

{\bf Case 1: Fat planks.}  

\vskip10pt

If $a \gg \delta$,  then we can adapt the high density argument from the last subsection to show that $U(\TT) \cap K$ fills a surprisingly large fraction of $K$ and hence $U(\TT) \cap B_a$ fills a surprisingly large fraction of $B_a$.  This requires some technical care,  but the high level ideas are the same as in the case $K=B$ above,  and the most ingredient is the high density lemma,  Lemma \ref{lemmahighdensity}.

\vskip10pt

This leaves the case when $a \approx \delta$,  so $K$ has dimensions roughly $\delta \times b \times r$.   When we study how the tubes $\TT_B[K]$ sit inside of $K$,  we essentially have a two-dimensional Kakeya problem.   Using the $L^2$ method,  it follows that $|U(\TT_B[K])| \gtrapprox |K|$, and so $\TT_B[K]$ essentially fills $K$.  

If $b \ll r$, then we call $K$ a thin plank, and if $b \approx r$, then we call $K$ a thin slab.  In the thin plank case, we have to study how the planks of  $\KK$ intersect each other.  Each plank $K \in \KK$ has a two-dimensional tangent plane,  spanned by the two longest axes of $K$.   When two planks intersect,  we call the intersection tangential if the two tangent planes are equal and we call the intersection transverse if the two tangent planes are transverse.    

\vskip10pt

{\bf Case 2: Thin planks intersecting transversely.}

\vskip10pt

When most intersections are transverse,  then the $L^2$ method gives strong bounds.   If we intersect a plank with a ball of radius $b$,  we get a slab.  Inside such a ball $B_b$,  we would see several slabs intersecting transversely.   Using the $L^2$ method,  it follows that $U(\KK)$ essentially fills $B_b$ and so $U(\TT)$ essentially fills $B_b$.  Then we can finish the proof as in Subsection \ref{subsecdensity}.

\vskip10pt

{\bf Case 3: Thin planks intersecting tangentially.}

\vskip10pt

In this case,  we will relate $\mu(\TT)$ to $\mu(\KK)$.   We will bound $\mu(\KK)$ by changing coordinates so that the planks become a set of tubes $\TT'$,  and we bound $\mu(\TT')$ using induction.   Here we sketch how to relate $\KK$ to $\TT'$.  

If all intersections are tangential,  then all the planks $K$ that intersect a given plank $K_0$ lie in a slab $S$ of dimensions $\frac{\delta r}{b} \times r \times r$.    In fact, each plank $K$ that intersects $K_0$ lies in $S$ and also has tangent plane close to the tangent plane of $S$.   We let $\KK_S$ be the set of planks $K$ so that $K \subset S$ and the tangent plane of $K$ is close to that of $S$.   When most intersections are tangential,  then $\mu(\KK) \approx \mu(\KK_S)$.

Now we can change coordinates so that $S$ becomes the unit cube and so each $K \in \KK_S$ becomes a tube $T'$ in the unit cube.   

$$ \textrm{ Planks $K \in \KK_S$ } \longleftrightarrow \textrm{ tubes in $B_1$}. $$

Figure \ref{fig:corr_plank_tube} illustrates this correspondence:

    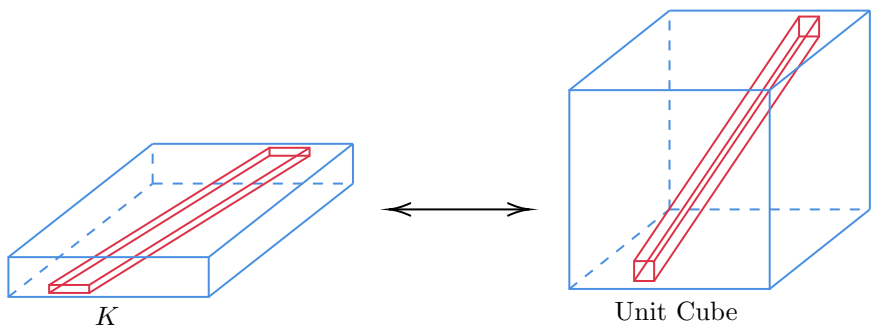
\begin{figure} [h!]
      \begin{center}
        \input{plank_tube_correspondence_diagram}
      \end{center}
      \caption{Correspondence between planks in $K$ and tubes in $B_1$}\label{fig:corr_plank_tube}
    \end{figure}

In Figure \ref{fig:corr_plank_tube}, the red plank on the left is a plank $K \subset S$.   Under the linear change of variables,  the plank $K$ corresponds to the red tube on the right.

We let $\TT'$ be the set of tubes on the right.   Then we have $\mu(\KK) \approx \mu(\KK_S) \approx  \mu(\TT')$.   We also have $\Delta_{max}(\TT') \approx \Delta_{max}(\KK_S) \le \Delta_{max}(\KK) \lessapprox 1$.    Therefore we can bound $\mu(\TT') \lessapprox |\TT'|^\beta$. 

With some additional computation,  this leads to a bound for $\mu(\TT)$ which gives the desired improvement $\mu(\TT) \ll |\TT|^\beta$.    This computation is similar to the proof of Lemma \ref{lemlowdenscase},  although a little more complicated.  We omit the details.

\vskip10pt

{\bf Case 4: Thin slabs.}

\vskip10pt

Finally we come to the thin slab case when $K$ has dimensions roughly $\delta \times r \times r$.  In the thin slab case,  we can get very strong bounds as long as $r$ is close to 1.   If $r$ is close to 1,  then the assumption $\Delta_{max}(\TT) \lessapprox 1$ guarantees that not too many tubes of $\TT$ can heavily intersect $K$,  and so it will follow that there are many thin slabs $K$.   The intersections of thin slabs with each other are well controlled by the $L^2$ method.  Therefore, if $r$ is close to 1, then $|U(\KK)|$ must be close to 1, and hence $|U(\TT)|$ must be close to 1 also.

\vskip10pt

Recall from the start of Section \ref{subsecsmallball} that we chose an angle $\rho \in [\delta, 1]$ with $|\TT_\rho| \gg \rho^{-2}$,  and the ball $B$ has radius $r = \delta/\rho$.   In order to make $r$ close to 1,  we need to choose $\rho$ close to $\delta$.   For this reason,  we need to know that $\TT$ is not sticky at scales $\rho$ very close to $\delta$.   Also,  when we fill in the details in Case 3,  we actually need to know that $\TT$ is not sticky at all scales $\rho$ with $\delta \ll \rho \ll 1$.  This was the reason that we focused on the not-sticky-at-all-scales case above.

\subsection{How do we reduce to the not-sticky-at-all-scales case?}

We saw in the argument above that it was important to reduce to the not-sticky-at-all-scales case.   In this subsection,  we explain how to do that.   We start by recalling the sticky case,  the not-sticky case,  and the not-sticky-at-all-scales case.

Recall that $\TT$ is a set of $\delta$-tubes in $B_1$ with $\Delta_{max}(\TT) \lessapprox 1$ and that $| \TT | \approx \delta^{-2}$. 

The sticky case means that for every $\rho \in [\delta,  1]$,  $\Delta_{max}(\TT_\rho) \lessapprox 1$.   Since $| \TT | \approx \delta^{-2}$,  this is equivalent to $| \TT[T_\rho]| \approx (\delta/\rho)^{-2}$.

If $\TT$ is not sticky,  it means that there is some $\rho \in [\delta, 1]$ so that $| \TT[T_\rho]| \ll (\delta/\rho)^{-2}$.   Such a $\rho$ must lie in the range $\delta \ll \rho \ll 1$.

We say that $\TT$ is not-sticky-at-all-scales if $| \TT[T_\rho]| \ll (\delta/\rho)^{-2}$ for every $\rho$ in the range $\delta \ll \rho \ll 1$.

Here is the rough idea how to reduce the not-sticky case to the not-sticky-at-all-scales case.   Suppose that there is some scale $\rho$ so that $| \TT[T_\rho]| \approx (\delta / \rho)^{-2}$.   It follows that $\Delta_{max}(\TT_\rho) \lessapprox 1$.   Now we can bound

$$ \mu(\TT) \lessapprox \mu(\TT[T_\rho]) \mu(\TT_\rho), $$

\noindent which reduces our original problem to two similar problems at smaller scales.   We try to keep reducing in this way.   If one of the smaller problems is not-sticky-at-all-scales  then we are stuck and we cannot reduce further.   Otherwise we can reduce further.   If we can keep reducing in this way to very small problems,  it means that our original set of tubes $\TT$ was sticky,  and we can handle it using the sticky Kakeya theorem.   Otherwise,  we get stuck with a problem that is not-sticky-at-all-scales  and we can handle it using the argument we have sketched in the last two sections.



\section{The high density lemma} \label{sechighdens}

We now come to the last ingredient in the proof that a worst case Kakeya set must be sticky: the proof of the high density lemma.  The Frostman condition plays a key role in the high density lemma,  so we recall the Frostman condition and put it in a slightly more general context.    Suppose that $\WW$ is a set of convex sets $W$ all lying in a given convex set $U$.   We say that $\WW$ is Frostman in $U$ if $\Delta_{max}(\WW) \approx \Delta(\WW,  U)$.  In the statement of the high density lemma, $\WW$ will be a set of tubes, but later in our discussion we will see more general convex sets.

\begin{lemma*} (High density lemma) Suppose that the exponent $\beta$ is as defined above.   Suppose $\TT$ is a set of $\delta$ tubes which is Frostman in $B_1$.   Then 

$$ |U(\TT)| \gtrapprox \left( \delta^2 |\TT| \right)^{1 - \beta}  \delta^{2 \beta}. $$

Equivalently,

$$ \mu(\TT) \lessapprox \left( \delta^2 |\TT| \right)^{1 - \beta}  \delta^{-2 \beta}.$$

\end{lemma*}

To digest this lemma we start with the case when $|\TT| \sim \delta^{-2}$.   If $|\TT| \sim \delta^{-2}$,  then $\Delta(\TT,  B_1) \sim 1$.   Since $\TT$ is Frostman,  $\Delta_{max}(\TT) \approx \Delta(\TT,  B_1) \sim 1$,  and so $|U(\TT)| \gtrapprox \delta^{2 \beta}$.    This matches the conclusion of the high-density lemma when $|\TT| \sim \delta^{-2}$.    The content of the high density lemma is that when $\TT$ is Frostman and $|\TT| \gg \delta^{-2}$,  $|U(\TT)|$ is much bigger than $\delta^{2 \beta}$.   

If $\TT$ is Frostman with $|\TT| \gg \delta^{-2}$,  it is easy to prove that $|U(\TT)| \gtrapprox \delta^{2 \beta}$.    Randomly decompose $\TT$ as a disjoint union $\TT = \sqcup_j \TT_j$,  where $| \TT_j | \sim \delta^{-2}$.   Since $\TT$ obeys the Frostman condition,  it is not hard to check that $\TT_j$ also obeys the Frostman condition.  Since $|\TT_j| \sim \delta^{-2}$,  we checked above that $|U(\TT_j)| \gtrapprox \delta^{2 \beta}$,  and so 

\begin{equation} \label{trivbound}
|U(\TT)| \gtrapprox |U(\TT_j)| \gtrapprox \delta^{2 \beta}.
\end{equation}

Even a small improvement on this trivial bound is enough to power the inductive argument in the last sections.   If (\ref{trivbound}) were sharp,  it would mean that for each $j$,  $|U(\TT)| \approx |U(\TT_j)|$.   This sounds intuitively unlikely: when $|\TT|$ is far bigger than $|\TT_j|$, we might expect $|U(\TT)|$ to be at least a little bigger than $|U(\TT_j)|$.   However,  it is not easy to prove this. 

In the proof,  we will work with the formulation involving $\mu(\TT)$.  The two formulations are equivalent because of the definition $\mu(\TT) = \frac{ |\TT| |T| }{|U(\TT)|}$.  Let $\gamma$ be the smallest exponent so that,  if $\TT$ is a Frostman set of tubes in $B_1$,  then

\begin{equation} \label{defgamma} \mu(\TT) \lessapprox (\delta^{-2})^\beta (\delta^2 | \TT|)^{\gamma}. \end{equation}

\noindent So Lemma \ref{lemmahighdensity} says that $\gamma = 1 - \beta$.   To power the inductive argument in the last sections,  we just need to prove that $\gamma < 1 $ (assuming that $\beta > 0$).   We say that $\TT$ is a worst-case example for Lemma \ref{lemmahighdensity} if

\begin{equation} \label{worstcase} \mu(\TT) \approx (\delta^{-2})^\beta (\delta^2 | \TT|)^{\gamma}. \end{equation}

\subsection{Looking for sticky Kakeya sets}

A crucial input to the proof is the sticky Kakeya theorem.   One scenario is that $\TT$ contains a sticky Kakeya set $\TT'$.   In this case,  $|U(\TT)| \ge |U(\TT')| \gtrapprox 1$,  and we are done.   When does $\TT$ contain a sticky Kakeya set?  Recall the definition of a sticky Kakeya set.   A set of tubes $\TT$ is a sticky Kakeya set if:

\begin{itemize}

\item $|\TT_\rho| \approx \rho^{-2}$ for all $\rho \in [\delta, 1]$

\item $\Delta_{max}(\TT) \lessapprox 1$.

\end{itemize}

To look for a sticky Kakeya set,  it helps to consider a dense sequence of scales $1 = \rho_0 > \rho_1 > ...  > \rho_N = \delta$.   We say the sequence of scales is dense if each quotient $\frac{\rho_{j-1}}{\rho_j}$ is very small.  We let $ \TT_{\rho_j} [ T_{\rho_{j-1}}]$ be the set of all $T_{\rho_j} \in \TT_{\rho_j}$ lying in the thicker tube $T_{\rho_{j-1}} \in \TT_{\rho_{j-1}}$.   The definition of sticky can be rephrased in terms of these sets $\TT_{\rho_j} [T_{\rho_{j-1}}]$.   If the sequence of scales is dense enough,  then $\TT$ is sticky if and only if for every $j$,

\begin{itemize}

\item $\left| \TT_{\rho_j} [T_{\rho_{j-1}}] \right| \approx \left( \frac{ \rho_{j-1} }{\rho_j} \right)^2$

\item $\Delta_{max} ( \TT_{\rho_j} [ T_{\rho_{j-1}}] ) \lessapprox 1$.    Equivalently,  $\TT_{\rho_j} [ T_{\rho_{j-1}} ]$ is Frostman.

\end{itemize}

Not every Frostman set of tubes in $B_1$ contains a sticky Kakeya set $\TT'$.   From the characterization of sticky Kakeya sets above,  we see that if $\TT$ does contain a sticky Kakeya set,  then for each $j$,  $\TT_{\rho_j} [T_{\rho_{j-1}}]$ must contain a Frostman set of tubes.   This can fail. 
On the other hand,  if each set $ \TT_{\rho_j} [T_{\rho_{j-1}}]$ is Frostman,  then $\TT$ does contain a sticky subset $\TT'$.   We state this as a lemma.

\begin{lemma} \label{lemfindsticky} If $\TT_{\rho_j}[ T_{\rho_{j-1}} ] $ is Frostman in $T_{\rho_{j-1}}$ for each $j$ and each $T_{\rho_{j-1}} \in \TT_{\rho_{j-1}}$,  then $\TT$ contains a subset $\TT'$ which is a sticky Kakeya set.   In fact,  we can even decompose $\TT$ as $\TT = \sqcup_j \TT_j$ where each $\TT_j$ is a sticky Kakeya set.
\end{lemma}

\begin{proof} [Proof sketch]  Since $\TT_{\rho_1}$ is Frostman,  $|\TT_{\rho_1}| \gtrapprox \rho_1^{-2}$.    Choose a random subset $\TT'_{\rho_1} \subset \TT_{\rho_1}$ with $|\TT'_{\rho_1}| \sim \rho_1^{-2}$.   Then $\TT'_{\rho_1}$ is also Frostman.

For each $T_{\rho_1} \in \TT'_{\rho_1}$,  note that $\TT_{\rho_2}[T_{\rho_1}]$ is Frostman and so $|\TT_{\rho_2} [T_{\rho_1}] \gtrapprox (\rho_1 / \rho_2)^2$.    Choose a random subset $\TT'_{\rho_2}[T_{\rho_1}] \subset \TT_{\rho_2} [T_{\rho_1}]$ with $|\TT'_{\rho_2}[T_{\rho_1}] | \sim (\rho_1/\rho_2)^2$.   Then $\TT'_{\rho_2}[T_{\rho_1}]$ is also Frostman.   We set $\TT'_{\rho_2} = \cup_{\rho_1 \in \TT'_{\rho_1}} \TT'_{\rho_2} [ T_{\rho_1}]$.  

Proceeding in this way,  we define $\TT'$,  and $\TT'$ is sticky because of the criterion above.  

If we choose random decompositions instead of just random subsets then we get a decomposition $\TT = \sqcup_j \TT_j$ where each $\TT_j$ is sticky.

\end{proof}

This raises the question whether we can find a dense sequence of scales $1 = \rho_0 > \rho_1 > ... > \rho_N = \delta$ so that for each $j$,  $ \TT_{\rho_j} [T_{\rho_{j-1}}]$ is Frostman in $T_{\rho_{j-1}}$.    We call such a sequence a good sequence of scales.  We can try to build a good sequence of scales by adding one scale at a time.   The process boils down to asking: is there a scale $\rho$ with $1 \gg \rho \gg \delta$ so that $\TT_\rho$ is Frostman and $\TT[T_\rho]$ is Frostman?   Call such a $\rho$ a good scale.

Since $\TT$ is Frostman in $B_1$,   it follows that $\TT_\rho$ is also Frostman in $B_1$.  The reason is that if $\Delta(\TT_\rho, K) \gg \Delta(\TT_\rho, B_1)$,  then it would follow that $\Delta(\TT, K) \gg \Delta(\TT,  B_1)$.  

But $\TT[T_\rho]$ is not necessarily Frostman.   Since $\TT$ is Frostman in $B_1$, we do know that for any convex set $K$,  $\Delta(\TT, K) \lessapprox \Delta(\TT,  B_1)$.   But if $\Delta(\TT,  T_\rho) \ll \Delta(\TT,  B_1)$,  then there could be some $K \subset T_\rho$ with $\Delta(\TT, K) \gg \Delta(\TT,  T_\rho)$.    

If $\TT[T_\rho]$ is Frostman, then we have a good scale and we can start to build a good sequence of scales which would lead to a sticky Kakeya set $\TT' \subset \TT$.  So we need to analyze the case when $\TT[T_\rho]$ is not Frostman and find some useful structure there.

\subsection{Analyzing the case that $\TT[T_\rho]$ is not Frostman}

If $\TT[T_\rho]$ is not Frostman, then it means by definition that there is some subset $W \subset T_\rho$ so that $\Delta(\TT[T_\rho], W) \gg \Delta(\TT[T_\rho], T_\rho)$.   As in Section \ref{secothershapes},  it is helpful to organize $\TT[T_\rho]$ by choosing sets $W$ that maximize $\Delta(\TT[T_\rho], W)$.  As in the beginning of Section \ref{secothershapes}, we can find a set $\WW(T_\rho)$ of such maximal $W$ by choosing them one at a time until each tube $T \in \TT[T_\rho]$ lies in $\approx 1$ $W \in \WW(T_\rho)$.  

We can assume that each $W \in \WW(T_\rho)$ has roughly the same dimensions: each $W$ has dimensions roughly $a \times b \times 1$, where $\delta \le a \le b \le \rho$.  We will see that if $\TT$ is a worst-case Kakeya set, then it strongly constrains the geometry of $W$.  

\begin{lemma} \label{lemWround} If $\TT$ is a worst-case example for Lemma \ref{lemmahighdensity} as in \eqref{worstcase} and $\WW[T_\rho]$ is defined as above, then $a \approx b$.  So each $W \in \WW(T_\rho)$ is approximately a tube $T_a$ of radius $a$ and length 1.
\end{lemma}

\begin{proof} [Proof sketch]

For each $T_\rho \in \TT[T_\rho]$, we have defined $\WW(T_\rho)$.  We gather all these sets together to form $\WW = \cup_{T_\rho \in \TT_\rho} \WW(T_\rho)$.  Now each $T \in \TT$ lies in $\approx 1$ $W \in \WW$.  Therefore we have $|\TT| \approx |\TT[W]| |\WW|$ and we can bound $\mu(\TT)$ by 

\begin{equation} \label{twoscales4} \mu(\TT) \lessapprox \mu(\TT[W]) \mu(\WW).
\end{equation}

Moreover, $\TT[W]$ and $\WW$ are each Frostman.  The set $\TT[W]$ is Frostman because we chose $W$ to maximize $\Delta(\TT[T_\rho], W)$.    To see that $\WW$ is Frostman,  we arrange by some pigeonholing arguments that the sets $\TT[W]$ have roughly the same cardinality.  Now if $K \subset B_1$ has $\Delta(\WW, K) \gg \Delta(\WW, B_1)$ it would imply $\Delta(\TT, K) \gg \Delta(\TT, B_1)$.   Since $\TT$ is Frostman,  it follows that $\WW$ is Frostman too.

Since $\TT[W]$ and $\WW$ are each Frostman, it almost looks like we could bound $\mu(\TT[W])$ and $\mu(\WW)$ by induction using \eqref{defgamma}.  That's not quite possible because \eqref{defgamma} applies to a set of tubes which is Frostman in a ball, whereas $\TT[W]$ is a set of tubes that is Frostman in a convex set $W$, and $\WW$ is a set of convex sets which is Frostman in a ball.  In order to use induction, we generalize Lemma \ref{lemmahighdensity} and \eqref{defgamma} to a set of convex sets which is Frostman in another convex set.  This more general version then applies to both $\TT[W]$ and $\WW$.

Suppose that $\VV$ is a set of convex sets which is Frostman in a convex set $U$.  By a linear change of variables, we can reduce to the case that $U = B_1$.  Then we examine the dimensions of $V \in \VV$.  Say each $V \in \VV$ has dimensions $a_{\VV} \times b_{\VV} \times 1$.   If $a_{\VV} \approx b_{\VV}$, then $V$ is a tube and we can apply \eqref{defgamma}.  The other extreme example is when $V$ is a slab of dimensions $a_{\VV} \times 1 \times 1$.   Sharp bounds for intersecting slabs have been known for a long time by the $L^2$ method.   So in the slab case, we get very strong bounds for $\mu(\VV)$.  This leaves an intermediate case when $V$ has dimensions $a_{\VV} \times b_{\VV} \times 1$ with $a_{\VV} \ll b_{\VV} \ll 1$.  In this case, the shape of $V$ is intermediate between a tube and a slab.  In fact, at large scales $V$ looks like a tube, but if we intersect $V$ with a small ball, then it looks like a slab.  
In this case, $\mu(\VV)$ can be bounded by a multiscale argument that combines estimates for tubes from \eqref{defgamma} with estimates for slabs from the $L^2$ method.

We apply this method to bound $\mu(\TT[W])$ and $\mu(\WW)$ and then plug those bounds into \eqref{twoscales4} to bound $\mu(\TT)$.  The computation shows that $\mu(\TT) \ll  (\delta^{-2})^\beta (\delta^2 | \TT|)^{\gamma}$ unless $a \approx b$.  In other words, if $\TT$ is a worst-case example for Lemma \ref{lemmahighdensity}, then $a \approx b$.  

Morally, the reason that the computation works out this way is that the bounds for slabs are sharp.  When we unwind the argument above, we bound $\mu(\TT)$ by applying \eqref{defgamma} at some scales and the $L^2$ bounds for slabs at other scales.  Since the bound for slabs is so strong, this improves on \eqref{defgamma} as long as we use the slab bound at some scales.  When $a \approx b$, then all the convex sets in the argument are tubes, and the slab bound never appers.  But if $a \ll b$, then each $W \in \WW$ looks like a plank and there are some scales where the slab bounds come into play.

\end{proof}

\subsection{Finding a good scale}

Lemma \ref{lemWround} leads to a general condition for finding a good scale.

\begin{lemma} \label{lemfindscale} Suppose that $\TT$ is a worst-case example for the high density lemma, as in \eqref{worstcase}.  If $|\TT| \gg \delta^{-2}$, then there is a good scale $a$.  Recall this means that

\begin{itemize}

\item $\delta \ll a \ll 1$.

\item $\TT[T_a]$ is Frostman and $\TT_a$ is Frostman

\end{itemize} 

Moreover,  $\TT[T_a]$ and $\TT_a$ will also be worst-case examples for the high density lemma.
\end{lemma}

\begin{proof} [Proof sketch] Since $|\TT| \gg \delta^{-2}$,  we can choose $\rho \ll 1$ so that $|\TT[T_\rho]| \gg (\rho/\delta)^2$ and so $\Delta(\TT,  T_\rho) \gg 1$.   

We define $\WW(T_\rho)$ as above.  Recall that each $W \in \WW(T_\rho)$ maximizes $\Delta(\TT[T_\rho], W)$, and so $\TT[W]$ is Frostman.  

Each $W$ has dimensions $a \times b \times 1$.  By Lemma \ref{lemWround}, $a \approx b$, and so $W$ is essentially a tube of radius $a$, $T_a$.  We also have $a \le \rho \ll 1$.  

Recall that $\WW = \cup_{T_\rho \in \TT_\rho} \WW(T_\rho)$, and that each $T \in \TT$ lies in $\approx 1$ set $W \in \WW$.  Since each $W$ is a tube of radius $a$, we must have $\WW = \TT_a$.

Since $\TT[W]$ is Frostman for each $W \in \WW$, we see that $\TT[T_a]$ is Frostman for each $T_a \in \TT_a$.  On the other hand, since $\TT$ is Frostman it implies that $\TT_a$ is Frostman.

Next we have to check that $a \gg \delta$.  Recall that we chose $\rho$ so that $\Delta(\TT[T_\rho], T_\rho) \gg 1$.  But $\Delta(\TT[W], W) \approx \Delta_{max}(\TT[T_\rho]) \ge \Delta(\TT[T_\rho], T_\rho) \gg 1$.  But if $a \approx \delta$,  then $W = T_a$ would essentially be a $\delta$-tube,  and then we would have
$\Delta(\TT[W], W) \approx 1$.  Therefore, we must have $a \gg \delta$ as desired.  This shows that $a$ is a good scale.

Finallly we sketch the proof that $\TT[T_a]$ and $\TT_a$ are both worst-case examples for the high density lemma.  Recall that 

$$ \mu(\TT) \lessapprox  \mu(\TT[T_a]) \mu(\TT_a). $$

\noindent Since $\TT[T_a]$ and $\TT_a$ are both Frostman,  we can bound $\mu(\TT[T_a])$ and $\mu(\TT_a)$ using \eqref{defgamma},  the definition of the exponent $\gamma$.   When we plug in and simplify the right-hand side,  we get $\delta^{- 2\beta} \left( \delta^2 |\TT| \right)^\gamma$,  and since $\TT$ is worst-case,  this is $\approx \mu(\TT)$.   Therefore,  every step in this chain must be an approximate equality.   In particular this means that $\mu(\TT[T_a])$ and $\mu(\TT_a)$ must be worst-case for the high density lemma.

\end{proof}

\subsection{Multiscale decomposition}

By applying Lemma \ref{lemfindscale} repeatedly, we can choose a sequence of scales $1 = \rho_0 \gg \rho_1 \gg \rho_2 \gg ... \gg \rho_N = \delta$ so that  $ \TT_{\rho_j}[ T_{\rho_{j-1}}]$ is always Frostman, and for each $j$ one of the following holds:

\begin{enumerate}

\item $\rho_{j-1} / \rho_{j}$ is very small, or 

\item  $\left|  \TT_{\rho_j} [  T_{\rho_{j-1}}] \right| \approx \left( \frac{ \rho_{j-1} }{\rho_j} \right)^2$.

\end{enumerate}

Figure \ref{fig:keyscales} is a picture showing how this sequence of scales may look:

    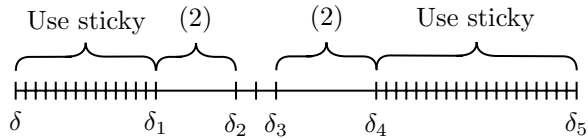
\begin{figure} [h!]
      \begin{center}
        \input{key_scales_diagram}
      \end{center}
      \caption{Key scales}\label{fig:keyscales}
    \end{figure}

Here each short vertical line represents a scale $\rho_j$.   These scales are generally quite close together,  except for two significant gaps.  Each significant gap must be in case (2),  and so we labelled them (2).

If $|\TT| \gg \delta^{-2}$,  not every interval can be in case 2 --  a definite fraction of intervals must be in case 1.   Since the intervals in case 1 are very small,  a definite fraction of scales must consist of many small intervals.   On any such block of small intervals,  we can apply sticky Kakeya,  giving a very strong bound.   In our picture,  we have drawn two such blocks of small intervals,  and they are labelled ``use sticky Kakeya''.

To bound $\mu(\TT)$,  we begin by factoring $\mu(\TT)$ into contributions coming from different scale ranges.   For instance,  in the scenario illustrated in figure 5,  we would bound $\mu(\TT)$ by




\begin{equation} \label{multscalefactor} \mu(\TT) \lessapprox \mu( \TT[T_{\delta_1}]) \mu(\TT_{\delta_1}[ T_{\delta_2}])  \mu(\TT_{\delta_2}[ T_{\delta_3}])  \mu(\TT_{\delta_3} [ T_{\delta_4}])  \mu(\TT_{\delta_4}) . \end{equation}

\noindent The five factors on the right-hand side correspond to five scale ranges in the picture.    Note that if we used the trivial bound (\ref{trivbound}) to bound each factor on the right-hand side,  then we would get back the trivial bound.   Instead,  we use sticky Kakeya on the scale ranges labelled sticky.  On the other scale ranges, we use the trivial bound.  But since $\beta > 0$,  the bound in the sticky case is better than the trivial bound,  and so our overall bound for $\mu(\TT)$ is better than the trivial bound (\ref{trivbound}).   A careful calculation gives the value $\gamma = 1 - \beta$.  

This finishes the outline of the proof of the high density lemma.

\vskip10pt

Notice that in this argument,  we broke the scales from $\delta$ to 1 into several ranges in a strategic way that was tailored to the geometry of $\TT$.   The idea of choosing these ranges strategically was introduced by Keleti and Shmerkin in \cite{KS},  and it has become a major tool in this circle of problems.   For instance,  it plays an important role in the solution of the Furstenberg set problem in \cite{KS} and \cite{RW}.

\section{Final recap}

At the beginning of the survey,  we said that the hero of the proof was multiscale analysis.   We begin with a worst-case Kakeya set $\TT$,  with $\mu(\TT) \approx |\TT|^\beta$.   We then relate $\TT$ to many other sets of tubes $\TT'$ obeying $\Delta_{max}(\TT') \lessapprox 1$.   By definition of $\beta$,  we know that $\mu(\TT') \lessapprox |\TT'|^\beta$,  and this gives us information about $\TT$.

As a final recap,  let us look back and see how we found all these sets of tubes $\TT'$.   We used several different ways to relate a set of tubes $\TT$ to a new set of tubes,  or more generally to a new set of convex sets $\WW$.   If we start with $\TT$,  we can

\begin{itemize}

\item Look at thicker tubes $\TT_\rho$.

\item Look at $\TT[T_\rho]$.

\item Intersect tubes of $\TT$ with a smaller ball $B$ and look at $\TT_B$.

\item Find convex sets $K$ that maximize $\Delta(\TT,  K)$.   Then we can look at $\TT[K]$.   Also,  we can form a set $\KK$ of these convex sets $K$ and look at $\KK$.   Sometimes we can change coordinates to convert $\KK$ to a set of tubes.

\end{itemize}

These operations can be chained together.   For instance,  starting with $\TT$ we might first look at $\TT_B$.   Then we might find convex sets $K$ that maximize $\Delta(\TT_B,  K)$ and study $\KK$.   After some coordinate changes,  we might be led to a new set of tubes $\TT'$.   Or starting with $\TT$ we might first look at $\TT_B$ and then look at a thicker set of tubes $\TT_{B, \rho}$.   Starting at $\TT$,  we may need to use several operations to arrive at a new set $\TT'$ that obeys $\Delta_{max}(\TT') \lessapprox 1$.    By combining information from many such sets of tubes $\TT'$,  the proof shows that if $\beta > 0$,  the set $\TT$ would have to have special geometric and algebraic structure,  closely matching the Heisenberg group.

We also note that the argument is essentially by induction.  If we unwind the induction,  then the argument effectively uses the above operations many times to get from the initial set of tubes $\TT$ to other sets of tubes $\TT'$.

At the beginning of the survey,  we said that the proof of the Kakeya problem is based on studying the problem at many scales.   Looking at a problem at many scales has been a central theme in harmonic analysis for a hundred years,  and this proof can be regarded as part of that tradition.   But when we said that the proof is based on studying the problem at many scales,  we really meant that the proof is based on bringing into play all the sets of tubes $\TT'$ that can be built from $\TT$ by the operations above.   This version of multiscale analysis is much newer.   It was built over the last twenty five years,  with important parts of the picture appearing just in the last few years.

To finish this essay,  let us return to the Katz-Zahl example.   As we mentioned in Section \ref{secstickynot},  Katz and Zahl found a cousin problem to the Kakeya problem where the analogue of Kakeya does not hold but the analogue of the sticky case appears likely to hold.   That example made me think it was unlikely that the Kakeya problem could be reduced to the sticky case.   Wang and Zahl did reduce the general Kakeya problem to the sticky case,  and so it is natural to ask why their method does not apply to the Katz-Zahl example.

The Katz-Zahl example concerns a cousin of the Kakeya problem where $\RR$ is replaced by the ring $ A= \FF_p[x] / (x^2)$.   The ring $A$ has a natural notion of distance with two distinct length scales.   If $a + bx \in A$ with $a,b \in \FF_p$,  we define

$$ \| a + b x \|_A := \begin{cases} 1 & \textrm{ if } a \not= 0 \\ p^{-1} &  \textrm{ if } a = 0,  b \not= 0 \\ 0  & \textrm{ if } a = b = 0 \\ \end{cases} $$

\noindent There is a cousin of the Heisenberg group in $A^3$ and it leads to a counterexample to the analogue of Theorem \ref{main}.   But unlike in $\CC^3$,  the Heisenberg group cousin in $A^3$ is {\it not} sticky.   It appears likely that in $A^3$,  the sticky case of Wolff axiom Kakeya conjecture is true,  but the general conjecture is false.

Looking back at the proof of the Kakeya conjecture, the key distinction between $\RR$ and the ring $A$ is that the ring $A$ has only two distinct non-zero scales.   The proof of the Kakeya conjecture requires discussing many scales in order to run the multiscale analysis.   In the ring $A$,  we do not have the rich range of related sets of tubes $\TT'$ that are used in the proof of the Kakeya conjecture.

\end{document}

%% file: intersecting_tubes_diagram.tex
\tikzset{every picture/.style={line width=0.75pt}} 

\begin{tikzpicture}[x=0.75pt,y=0.75pt,yscale=-1,xscale=1]

\draw  [color={rgb, 255:red, 74; green, 144; blue, 226 }  ,draw opacity=1 ] (41.41,175.16) -- (209.56,14.59) -- (262.77,70.31) -- (94.62,230.88) -- cycle ;
\draw  [color={rgb, 255:red, 74; green, 144; blue, 226 }  ,draw opacity=1 ] (40.52,61.72) -- (260.55,136.84) -- (235.66,209.75) -- (15.63,134.63) -- cycle ;
\draw  [color={rgb, 255:red, 220; green, 50; blue, 71 }  ,draw opacity=1 ][fill={rgb, 255:red, 220; green, 50; blue, 71 }  ,fill opacity=0.14 ] (73.97,203.78) -- (208.16,18.66) -- (216.73,24.87) -- (82.53,209.99) -- cycle ;
\draw  [color={rgb, 255:red, 220; green, 50; blue, 71 }  ,draw opacity=1 ][fill={rgb, 255:red, 220; green, 50; blue, 71 }  ,fill opacity=0.14 ] (89.72,214.47) -- (207.1,18.26) -- (216.18,23.69) -- (98.79,219.9) -- cycle ;
\draw  [color={rgb, 255:red, 220; green, 50; blue, 71 }  ,draw opacity=1 ][fill={rgb, 255:red, 220; green, 50; blue, 71 }  ,fill opacity=0.14 ] (50.68,174.91) -- (247.42,58.41) -- (252.81,67.51) -- (56.07,184.01) -- cycle ;
\draw  [color={rgb, 255:red, 220; green, 50; blue, 71 }  ,draw opacity=1 ][fill={rgb, 255:red, 220; green, 50; blue, 71 }  ,fill opacity=0.14 ] (47.13,64.31) -- (234.43,195.45) -- (228.36,204.11) -- (41.07,72.97) -- cycle ;
\draw  [color={rgb, 255:red, 220; green, 50; blue, 71 }  ,draw opacity=1 ][fill={rgb, 255:red, 220; green, 50; blue, 71 }  ,fill opacity=0.14 ] (28.77,100.59) -- (254.43,137.39) -- (252.73,147.82) -- (27.06,111.03) -- cycle ;
\draw  [color={rgb, 255:red, 220; green, 50; blue, 71 }  ,draw opacity=1 ][fill={rgb, 255:red, 220; green, 50; blue, 71 }  ,fill opacity=0.14 ] (23.16,120.23) -- (244.91,175.93) -- (242.34,186.19) -- (20.58,130.49) -- cycle ;

\end{tikzpicture}

%% file: tube_intersect_ball_diagram.tex
\tikzset{every picture/.style={line width=0.75pt}} 

\begin{tikzpicture}[x=0.75pt,y=0.75pt,yscale=-1,xscale=1]

\draw  [color={rgb, 255:red, 74; green, 144; blue, 226 }  ,draw opacity=1 ] (46.16,25.2) -- (264.48,93.12) -- (261.34,103.22) -- (43.01,35.3) -- cycle ;
\draw  [color={rgb, 255:red, 74; green, 144; blue, 226 }  ,draw opacity=1 ] (43.23,43.32) -- (269.72,74.62) -- (268.27,85.1) -- (41.78,53.79) -- cycle ;
\draw  [color={rgb, 255:red, 74; green, 144; blue, 226 }  ,draw opacity=1 ] (44.34,43.11) -- (255.35,131.16) -- (251.28,140.92) -- (40.27,52.87) -- cycle ;
\draw  [color={rgb, 255:red, 74; green, 144; blue, 226 }  ,draw opacity=1 ] (41.88,64.04) -- (265.85,110.01) -- (263.72,120.37) -- (39.75,74.4) -- cycle ;
\draw   (104,79.17) .. controls (104,54.22) and (124.22,34) .. (149.17,34) .. controls (174.11,34) and (194.33,54.22) .. (194.33,79.17) .. controls (194.33,104.11) and (174.11,124.33) .. (149.17,124.33) .. controls (124.22,124.33) and (104,104.11) .. (104,79.17) -- cycle ;
\draw  [color={rgb, 255:red, 220; green, 50; blue, 71 }  ,draw opacity=1 ][fill={rgb, 255:red, 220; green, 50; blue, 71 }  ,fill opacity=0.41 ] (116.78,51.39) -- (189.69,66.88) -- (187.82,75.72) -- (114.91,60.23) -- cycle ;
\draw  [color={rgb, 255:red, 220; green, 50; blue, 71 }  ,draw opacity=1 ][fill={rgb, 255:red, 220; green, 50; blue, 71 }  ,fill opacity=0.41 ] (110.07,75.31) -- (181.32,97.19) -- (178.68,105.8) -- (107.43,83.92) -- cycle ;

\end{tikzpicture}

%% file: tube_intersect_ball_two_diagram.tex
\tikzset{every picture/.style={line width=0.75pt}} 

\begin{tikzpicture}[x=0.75pt,y=0.75pt,yscale=-1,xscale=1]

\draw  [color={rgb, 255:red, 74; green, 144; blue, 226 }  ,draw opacity=0.48 ] (272.96,42.94) -- (149.17,235.18) -- (144.96,232.47) -- (268.75,40.24) -- cycle ;
\draw  [color={rgb, 255:red, 74; green, 144; blue, 226 }  ,draw opacity=0.48 ] (264.17,35.13) -- (163.44,240.39) -- (158.95,238.18) -- (259.68,32.92) -- cycle ;
\draw  [color={rgb, 255:red, 74; green, 144; blue, 226 }  ,draw opacity=0.48 ] (293.43,217.98) -- (152.57,37.88) -- (156.51,34.8) -- (297.37,214.9) -- cycle ;
\draw  [color={rgb, 255:red, 74; green, 144; blue, 226 }  ,draw opacity=0.48 ] (303.7,212.25) -- (142.96,49.65) -- (146.51,46.13) -- (307.25,208.74) -- cycle ;
\draw  [color={rgb, 255:red, 74; green, 144; blue, 226 }  ,draw opacity=0.48 ] (231.25,240.32) -- (211.01,12.57) -- (215.99,12.13) -- (236.23,239.87) -- cycle ;
\draw  [color={rgb, 255:red, 74; green, 144; blue, 226 }  ,draw opacity=0.48 ] (242.99,241.09) -- (196.55,17.22) -- (201.44,16.2) -- (247.89,240.08) -- cycle ;
\draw  [color={rgb, 255:red, 74; green, 144; blue, 226 }  ,draw opacity=0.48 ] (128.16,156) -- (352.35,111.08) -- (353.33,115.98) -- (129.15,160.9) -- cycle ;
\draw  [color={rgb, 255:red, 74; green, 144; blue, 226 }  ,draw opacity=0.48 ] (128.67,167.75) -- (346.16,97.21) -- (347.7,101.96) -- (130.21,172.5) -- cycle ;
\draw   (184,138.17) .. controls (184,113.22) and (204.22,93) .. (229.17,93) .. controls (254.11,93) and (274.33,113.22) .. (274.33,138.17) .. controls (274.33,163.11) and (254.11,183.33) .. (229.17,183.33) .. controls (204.22,183.33) and (184,163.11) .. (184,138.17) -- cycle ;
\draw  [color={rgb, 255:red, 220; green, 50; blue, 71 }  ,draw opacity=1 ][fill={rgb, 255:red, 220; green, 50; blue, 71 }  ,fill opacity=0.41 ] (194.76,159.24) -- (231.82,94.56) -- (235.29,96.55) -- (198.23,161.22) -- cycle ;
\draw  [color={rgb, 255:red, 220; green, 50; blue, 71 }  ,draw opacity=1 ][fill={rgb, 255:red, 220; green, 50; blue, 71 }  ,fill opacity=0.41 ] (205.65,101.71) -- (262.81,167.46) -- (259.79,170.08) -- (202.63,104.33) -- cycle ;
\draw  [color={rgb, 255:red, 220; green, 50; blue, 71 }  ,draw opacity=1 ][fill={rgb, 255:red, 220; green, 50; blue, 71 }  ,fill opacity=0.41 ] (220.84,95) -- (233.04,181.26) -- (229.08,181.82) -- (216.88,95.56) -- cycle ;
\draw  [color={rgb, 255:red, 220; green, 50; blue, 71 }  ,draw opacity=1 ][fill={rgb, 255:red, 220; green, 50; blue, 71 }  ,fill opacity=0.41 ] (271.48,129.82) -- (187.06,151.34) -- (186.07,147.47) -- (270.49,125.95) -- cycle ;

\end{tikzpicture}

%% file: plank_tube_correspondence_diagram.tex
\tikzset{every picture/.style={line width=0.75pt}} 

\begin{tikzpicture}[x=0.75pt,y=0.75pt,yscale=-1,xscale=1]

\draw [color={rgb, 255:red, 74; green, 144; blue, 226 }  ,draw opacity=1 ] [dash pattern={on 3.75pt off 3.75pt}]  (370,190) -- (420,150) ;
\draw [color={rgb, 255:red, 74; green, 144; blue, 226 }  ,draw opacity=1 ] [dash pattern={on 3.75pt off 3.75pt}]  (420,150) -- (520,150) ;
\draw [color={rgb, 255:red, 74; green, 144; blue, 226 }  ,draw opacity=1 ] [dash pattern={on 3.75pt off 3.75pt}]  (420,150) -- (420,50) ;
\draw [color={rgb, 255:red, 220; green, 50; blue, 71 }  ,draw opacity=1 ]   (402.33,176) -- (484.67,53) ;
\draw [color={rgb, 255:red, 220; green, 50; blue, 71 }  ,draw opacity=1 ]   (412.33,176) -- (494.67,53) ;
\draw [color={rgb, 255:red, 220; green, 50; blue, 71 }  ,draw opacity=1 ]   (402.33,186) -- (484.67,63) ;
\draw [color={rgb, 255:red, 220; green, 50; blue, 71 }  ,draw opacity=1 ]   (412.33,186) -- (494.67,63) ;
\draw [color={rgb, 255:red, 74; green, 144; blue, 226 }  ,draw opacity=1 ] [dash pattern={on 3.75pt off 3.75pt}]  (162,137) -- (262,137) ;
\draw [color={rgb, 255:red, 74; green, 144; blue, 226 }  ,draw opacity=1 ]   (90,174) -- (162,117) ;
\draw [color={rgb, 255:red, 74; green, 144; blue, 226 }  ,draw opacity=1 ]   (190,174) -- (262,117) ;
\draw [color={rgb, 255:red, 74; green, 144; blue, 226 }  ,draw opacity=1 ]   (190,194) -- (262,137) ;
\draw [color={rgb, 255:red, 74; green, 144; blue, 226 }  ,draw opacity=1 ] [dash pattern={on 3.75pt off 3.75pt}]  (90,194) -- (162,137) ;
\draw [color={rgb, 255:red, 74; green, 144; blue, 226 }  ,draw opacity=1 ]   (162,117) -- (167.22,117) -- (262,117) ;
\draw [color={rgb, 255:red, 74; green, 144; blue, 226 }  ,draw opacity=1 ]   (262,117) -- (262,137) ;
\draw [color={rgb, 255:red, 74; green, 144; blue, 226 }  ,draw opacity=1 ] [dash pattern={on 3.75pt off 3.75pt}]  (162,137) -- (162,117) ;
\draw  [color={rgb, 255:red, 220; green, 50; blue, 71 }  ,draw opacity=1 ] (110.33,188) -- (130.33,188) -- (130.33,192) -- (110.33,192) -- cycle ;
\draw  [color={rgb, 255:red, 220; green, 50; blue, 71 }  ,draw opacity=1 ] (220.33,119) -- (240.33,119) -- (240.33,123) -- (220.33,123) -- cycle ;
\draw [color={rgb, 255:red, 220; green, 50; blue, 71 }  ,draw opacity=1 ]   (110.33,188) -- (220.33,119) ;
\draw [color={rgb, 255:red, 220; green, 50; blue, 71 }  ,draw opacity=1 ]   (110.33,192) -- (220.33,123) ;
\draw [color={rgb, 255:red, 220; green, 50; blue, 71 }  ,draw opacity=1 ]   (130.33,188) -- (240.33,119) ;

\draw [color={rgb, 255:red, 220; green, 50; blue, 71 }  ,draw opacity=1 ]   (130.33,192) -- (240.33,123) ;

\draw  [color={rgb, 255:red, 74; green, 144; blue, 226 }  ,draw opacity=1 ] (90,174) -- (190,174) -- (190,194) -- (90,194) -- cycle ;
\draw  [color={rgb, 255:red, 74; green, 144; blue, 226 }  ,draw opacity=1 ] (370,90) -- (470,90) -- (470,190) -- (370,190) -- cycle ;
\draw [color={rgb, 255:red, 74; green, 144; blue, 226 }  ,draw opacity=1 ]   (370,90) -- (420,50) ;
\draw [color={rgb, 255:red, 74; green, 144; blue, 226 }  ,draw opacity=1 ]   (470,90) -- (520,50) ;
\draw [color={rgb, 255:red, 74; green, 144; blue, 226 }  ,draw opacity=1 ]   (470,190) -- (520,150) ;
\draw [color={rgb, 255:red, 74; green, 144; blue, 226 }  ,draw opacity=1 ]   (520,50) -- (420,50) ;
\draw [color={rgb, 255:red, 74; green, 144; blue, 226 }  ,draw opacity=1 ]   (520,150) -- (520,50) ;
\draw    (282,150) -- (348,150) ;
\draw [shift={(350,150)}, rotate = 180] [color={rgb, 255:red, 0; green, 0; blue, 0 }  ][line width=0.75]    (10.93,-3.29) .. controls (6.95,-1.4) and (3.31,-0.3) .. (0,0) .. controls (3.31,0.3) and (6.95,1.4) .. (10.93,3.29)   ;
\draw [shift={(280,150)}, rotate = 0] [color={rgb, 255:red, 0; green, 0; blue, 0 }  ][line width=0.75]    (10.93,-3.29) .. controls (6.95,-1.4) and (3.31,-0.3) .. (0,0) .. controls (3.31,0.3) and (6.95,1.4) .. (10.93,3.29)   ;
\draw  [color={rgb, 255:red, 220; green, 50; blue, 71 }  ,draw opacity=1 ] (402.33,176) -- (412.33,176) -- (412.33,186) -- (402.33,186) -- cycle ;
\draw  [color={rgb, 255:red, 220; green, 50; blue, 71 }  ,draw opacity=1 ] (484.67,53) -- (494.67,53) -- (494.67,63) -- (484.67,63) -- cycle ;

\draw (139.25,197.4) node [anchor=north] [inner sep=0.75pt]    {$K$};
\draw (423.34,195) node [anchor=north] [inner sep=0.75pt]   [align=left] {Unit Cube};

\end{tikzpicture}

%% file: key_scales_diagram.tex
\tikzset{every picture/.style={line width=0.75pt}} 

\begin{tikzpicture}[x=0.75pt,y=0.75pt,yscale=-1,xscale=1]

\draw    (120,100) -- (50,100) (115,104.5) -- (115,95.5)(110,104.5) -- (110,95.5)(105,104.5) -- (105,95.5)(100,104.5) -- (100,95.5)(95,104.5) -- (95,95.5)(90,104.5) -- (90,95.5)(85,104.5) -- (85,95.5)(80,104.5) -- (80,95.5)(75,104.5) -- (75,95.5)(70,104.5) -- (70,95.5)(65,104.5) -- (65,95.5)(60,104.5) -- (60,95.5)(55,104.5) -- (55,95.5) ;
\draw [shift={(50,100)}, rotate = 360] [color={rgb, 255:red, 0; green, 0; blue, 0 }  ][line width=0.75]    (0,4.47) -- (0,-4.47)   ;
\draw [shift={(120,100)}, rotate = 360] [color={rgb, 255:red, 0; green, 0; blue, 0 }  ][line width=0.75]    (0,4.47) -- (0,-4.47)   ;
\draw    (180,100) -- (160,100) (170,104.5) -- (170,95.5) ;
\draw [shift={(160,100)}, rotate = 360] [color={rgb, 255:red, 0; green, 0; blue, 0 }  ][line width=0.75]    (0,4.47) -- (0,-4.47)   ;
\draw [shift={(180,100)}, rotate = 360] [color={rgb, 255:red, 0; green, 0; blue, 0 }  ][line width=0.75]    (0,4.47) -- (0,-4.47)   ;
\draw    (330,100) -- (230,100) (325,104.5) -- (325,95.5)(320,104.5) -- (320,95.5)(315,104.5) -- (315,95.5)(310,104.5) -- (310,95.5)(305,104.5) -- (305,95.5)(300,104.5) -- (300,95.5)(295,104.5) -- (295,95.5)(290,104.5) -- (290,95.5)(285,104.5) -- (285,95.5)(280,104.5) -- (280,95.5)(275,104.5) -- (275,95.5)(270,104.5) -- (270,95.5)(265,104.5) -- (265,95.5)(260,104.5) -- (260,95.5)(255,104.5) -- (255,95.5)(250,104.5) -- (250,95.5)(245,104.5) -- (245,95.5)(240,104.5) -- (240,95.5)(235,104.5) -- (235,95.5) ;
\draw [shift={(230,100)}, rotate = 360] [color={rgb, 255:red, 0; green, 0; blue, 0 }  ][line width=0.75]    (0,4.47) -- (0,-4.47)   ;
\draw [shift={(330,100)}, rotate = 360] [color={rgb, 255:red, 0; green, 0; blue, 0 }  ][line width=0.75]    (0,4.47) -- (0,-4.47)   ;
\draw    (120,100) -- (160,100) ;
\draw    (180,100) -- (230,100) ;
\draw   (119.92,89.79) .. controls (119.91,85.12) and (117.57,82.8) .. (112.9,82.81) -- (94.97,82.88) .. controls (88.3,82.91) and (84.96,80.59) .. (84.95,75.92) .. controls (84.96,80.59) and (81.64,82.93) .. (74.97,82.95)(77.97,82.94) -- (57.05,83.02) .. controls (52.38,83.03) and (50.06,85.37) .. (50.08,90.04) ;
\draw   (159.9,90) .. controls (159.9,85.33) and (157.57,83) .. (152.9,83) -- (150.03,83) .. controls (143.36,83) and (140.03,80.67) .. (140.03,76) .. controls (140.03,80.67) and (136.7,83) .. (130.03,83)(133.03,83) -- (127.16,83) .. controls (122.49,83) and (120.16,85.33) .. (120.16,90) ;
\draw   (230.2,89.75) .. controls (230.2,85.08) and (227.87,82.75) .. (223.2,82.75) -- (215.14,82.75) .. controls (208.47,82.75) and (205.14,80.42) .. (205.14,75.75) .. controls (205.14,80.42) and (201.81,82.75) .. (195.14,82.75)(198.14,82.75) -- (187.09,82.75) .. controls (182.42,82.75) and (180.09,85.08) .. (180.09,89.75) ;
\draw   (330.05,89.75) .. controls (330.08,85.08) and (327.76,82.74) .. (323.09,82.71) -- (290.09,82.55) .. controls (283.42,82.52) and (280.1,80.17) .. (280.12,75.5) .. controls (280.1,80.17) and (276.76,82.48) .. (270.09,82.45)(273.09,82.46) -- (237.09,82.28) .. controls (232.42,82.25) and (230.08,84.57) .. (230.05,89.24) ;

\draw (85.34,73.25) node [anchor=south] [inner sep=0.75pt]   [align=left] {Use sticky};
\draw (279.6,71.5) node [anchor=south] [inner sep=0.75pt]   [align=left] {Use sticky};
\draw (139.75,72.26) node [anchor=south] [inner sep=0.75pt]    {$( 2)$};
\draw (205.38,72.01) node [anchor=south] [inner sep=0.75pt]    {$( 2)$};
\draw (50,108.87) node [anchor=north] [inner sep=0.75pt]    {$\delta $};
\draw (120,109.12) node [anchor=north] [inner sep=0.75pt]    {$\delta _{1}$};
\draw (160,109.87) node [anchor=north] [inner sep=0.75pt]    {$\delta _{2}$};
\draw (180,110.12) node [anchor=north] [inner sep=0.75pt]    {$\delta _{3}$};
\draw (230,109.87) node [anchor=north] [inner sep=0.75pt]    {$\delta _{4}$};
\draw (330,109.37) node [anchor=north] [inner sep=0.75pt]    {$\delta _{5}$};

\end{tikzpicture}